\documentclass[a4paper,11pt, leqno]{article}
\usepackage{amssymb, amsmath, amsthm, graphicx}
\usepackage{color, epsfig, amsfonts}
\usepackage{mathrsfs}
\usepackage{breqn}

\newtheorem{thm}{Theorem}[section]
\newtheorem{lem}[thm]{Lemma}
\newtheorem{cor}[thm]{Corollary}
\newtheorem{prop}[thm]{Proposition}
\newtheorem{defn}{Definition}[section]

\newtheorem{rem}{Remark}[section]

\numberwithin{equation}{section}

\renewcommand{\a}{\alpha}
\renewcommand{\b}{\beta}
\newcommand{\e}{\varepsilon}
\newcommand{\de}{\delta}
\newcommand{\fa}{\varphi}
\newcommand{\ga}{\gamma}

\newcommand{\la}{\lambda}
\renewcommand{\th}{\theta}
\newcommand{\si}{\sigma}

\renewcommand{\t}{\tau}

\newcommand{\De}{\Delta}
\newcommand{\Ga}{\Gamma}

\newcommand{\lan}{\langle}
\newcommand{\ran}{\rangle}
\newcommand{\na}{\nabla}

\def\R{{\mathbb{R}}}
\def\N{{\mathbb{N}}}

\def\T{{\mathbb{T}}}

\allowdisplaybreaks

\title{Global solvability and convergence to stationary solutions\\
in singular quasilinear stochastic PDEs}

\author{Tadahisa Funaki$^\ast$ and Bin Xie$^\star$}

\date{}

\begin{document}
\maketitle

\begin{abstract}   \noindent
We consider singular quasilinear stochastic partial differential equations (SPDEs)
studied in \cite{FHSX}, which are defined in paracontrolled sense.
The main aim of the present article is to establish the global-in-time solvability 
for a particular class of SPDEs with origin in particle systems and, 
under a certain additional condition on the noise, prove 
the convergence of the solutions to stationary solutions as $t\to\infty$.
We apply the method of energy inequality and Poincar\'e inequality.
It is essential that the Poincar\'e constant can be taken uniformly
in an approximating sequence of the noise.  We also use the continuity of
the solutions in the enhanced noise, initial values and coefficients of the equation,
which we prove in this article for general SPDEs discussed in \cite{FHSX}
except that in the enhanced noise.  Moreover, we 
apply the initial layer property of improving regularity of the solutions in a short time.
\end{abstract}
\footnote{ \hskip -6.5mm
$^\ast$Department of Mathematics,
School of Fundamental Science and Engineering,
Waseda University,
3-4-1 Okubo, Shinjuku-ku,
Tokyo 169-8555, Japan.
e-mail: funaki@ms.u-tokyo.ac.jp \\
$^\star$Department of Mathematical Sciences,
Faculty of Science, Shinshu University,
3-1-1 Asahi, Matsumoto, Nagano 390-8621, Japan.
e-mail: bxie@shinshu-u.ac.jp; bxieuniv@outlook.com
}

\section{Introduction}\label{sec:1}

We studied in \cite{FHSX} the following quasilinear stochastic partial differential equation
(SPDE) defined in paracontrolled sense
\begin{equation}  \label{eq:1.1}
\partial_t u = a(\nabla u)\De u + g(\nabla u)\cdot\xi,
\end{equation}
on one dimensional torus $\T\simeq [0,1)$ having the spatial noise
$\xi\in C^{\a-2}, \a\in (\frac43,\frac32)$, where $\nabla=\partial_x$, $\De=\partial_x^2$
and $C^{\a}\equiv C^{\a}(\T)=B_{\infty,\infty}^\a(\T)$ denotes the 
H\"older-Besov space on $\T$ with regularity exponent $\a\in \R$
equipped with the norm $\|\cdot\|_{C^\a}$.
We showed the local-in-time solvability and the continuity of the solution 
in the enhanced noise $\hat\xi$.  More precisely, assuming that the
coefficients satisfy $a, g\in C_b^3(\R)$ and 
\begin{equation}  \label{eq:acC}
c_-\le a(v)\le c_+,
\end{equation}
for some $c_-,c_+>0$, it was shown that \eqref{eq:1.1} with the initial value $u_0\in C^\a$
has a solution $u$ up to some $T_*>0$ (see \eqref{eq:Tstar})
and if the enhanced noise $\hat\xi=(\xi, \Pi(\nabla X,\xi))$ converges
in $C^{\a-2}\times C^{2\a-3}$, then the corresponding solution 
$u= u^{\hat\xi}$ converges in $\mathcal{L}_T^\a$,
where $\mathcal{L}_T^\a =C([0,T],C^\a)\cap C^{\a/2}([0,T],L^\infty)$
taking $T>0$ uniformly in a neighborhood of some $\hat\xi$ and $L^\infty=L^\infty(\T)$.
Here $X=(-\De)^{-1}(\xi-\xi(\T)), \xi(\T)\equiv \lan \xi,1\ran$, and
$\Pi(\nabla X,\xi)$ denotes the resonant term in the paraproduct of 
$\nabla X$ and $\xi$.  In particular, we see $u(t) \in \cap_{\de>0}C^{\frac32-\de}$
for $t<T_*$, see \cite{FHSX} for details.
We remark that the same result  still holds under a weaker assumption:
$a, g\in C^3(\R)$ satisfying \eqref{eq:acC}, see Lemma \ref{lem:1.4} below.

\subsection{The aim of the article}

The present article is a continuation of \cite{FHSX}.  Assuming
that two coefficients $a$ and $g$ satisfy the relation $a=g'$,
we establish the global-in-time solvability and convergence of the
solution to a stationary solution as $t\to\infty$, see Theorems \ref{thm:1.1} 
and \ref{thm:1.1-u(t)}.  For the convergence to the stationary solution, we assume
$|\mu_\xi|$ is small enough for the noise $\xi$, where $\mu_\xi$ is the constant
defined from $\xi$ by \eqref{eq:th+A1} below.
A typical example of the noise $\xi$ is the derivative of a periodic Brownian
motion $w=w(x), x\in\T$: $\xi=\dot{w}$ and, for this $\xi$, $\mu_\xi=0$ holds.
Then, in general without assuming such conditions for the coefficients $a, g$ 
and noise $\xi$, we show the continuity of the local-in-time solution in
initial values, see Theorem \ref{thm:CIN}.  This is used for the proofs of 
Theorems \ref{thm:1.1} and \ref{thm:1.1-u(t)}.

Let  $\fa\in C^4(\R)$ satisfying 
\begin{align} \label{c-fa}
 c_-\le \fa'(v)\le c_+
\end{align}
for some $c_-, c_+>0$ and $\chi \in C^3(\R)$
be given, and consider the SPDE
\begin{equation}  \label{eq:2}
\partial_t v = \De \{\fa(v)\} + \nabla\{\chi(v) \xi\}, \quad x \in \T.
\end{equation}
Note that the linear case $\fa(v)=\chi(v)=v$ is included.
Then, for every $m\in \R$, if $u\equiv u_m$ is a solution of \eqref{eq:1.1} 
with $a(v)=\fa'(v+m)$, $g(v)=\chi(v+m)$, 
$v:=\nabla u_m+m$ solves the equation \eqref{eq:2} 
under the assumption $a, g\in C_b^3(\R)$, 
see Section 1.2 of \cite{FHSX} and an indirect Definition \ref{defn:sol-1.2}
of the solution of \eqref{eq:2} below.  Note that, if $\xi$ is smooth, this is true 
in classical sense.  In particular, \eqref{eq:2} has a local-in-time solution $v$ in 
$\mathcal{L}_T^{\a-1}$.  See Subsection \ref{sec:6.1} for
more explanation on the relation between $u$ and $v$ 
under the weak assumption $a, g\in C^3(\R)$.  
  We note that
the equation \eqref{eq:2} has a mass conservation law:
\begin{equation} \label{eq:mass}
\int_\T v(t,x)dx = m
\end{equation}
for all $t\ge 0$ with a constant $m\in\R$, which is determined
from its initial value $v_0$. 

As we mentioned above, the main aim of the present article is to show the 
global-in-time solvability and establish the
convergence of the solution to the stationary solution as $t\to\infty$ for
the SPDE \eqref{eq:1.1} when the coefficients satisfy $a=g'$.  For this purpose,
it turns out to be more convenient to study the SPDE in the form of \eqref{eq:2}.
Due to the discussions in Subsection \ref{sec:6.1}, especially by
Definition \ref{defn:sol-1.2}, the result
for \eqref{eq:2} implies that for $\nabla u$ for the solution $u$ of \eqref{eq:1.1}.
Our assumption $a=g'$
for \eqref{eq:1.1} corresponds to $\chi=\fa$ for \eqref{eq:2} so that  we consider the 
equation of the special form
\begin{equation}  \label{eq:0-v}
\partial_t v = \De \{\fa(v)\} + \nabla\{\fa(v) \xi\}, 
\end{equation}
on $\T$.  The equation \eqref{eq:0-v} has a physical meaning in the sense
that it can be derived from a microscopic particle system in random
environment, see \cite{LSX}.

Our another aim is to establish the continuity of the solution of \eqref{eq:1.1}
in initial values in general without assuming $a=g'$.  We use such property 
in the study of the SPDE \eqref{eq:0-v}, but we show it for the general case. 
The solution determines a continuous flow on the state space
$C^\a\cup\{\De\}$ with a death point $\De$ added to cover
the explosion of the solutions, see Remark \ref{rem-2.4-0} in Subsection \ref{sec:2.3}.
The equation \eqref{eq:1.1}
can be considered as if a deterministic PDE once $\xi$ is fixed, so that it
is different from the SPDEs driven by the space-time white noise, but the
continuity in initial value $u_0$ in our case corresponds to the strong Feller property
for such SPDEs, cf.\ \cite{HM}.


\subsection{Global solution in time and convergence to stationary solutions}

We consider the SPDE \eqref{eq:0-v} with $\xi\in C^{\a-2}, \a\in
(\frac43,\frac32)$.  To describe its stationary solutions, for a given $\xi$,
we define its integral $\eta(x):= \lan\xi,1_{[0,x]}\ran \equiv \int_0^x \xi(y)dy, x \in \T$.
Note that $\eta$
is not periodic, but $\tilde\eta(x):= \lan\xi-\si,1_{[0,x]}\ran = \eta(x)-\si x$
is periodic, where 
$$\si\equiv \si_\xi :=  \xi(\T)=
\lan \xi, 1\ran =\eta(1)\in \R.
$$
It is known that $\eta\in C^{\a-1}(\R)$ and especially $\eta\in C([0,1])$,
see Lemma A.10 of \cite{GIP-15}. 
Typically, we can take $\xi=\dot{w}(x)+\si$ with a periodic Brownian motion
$w(x)$, $x\in \T$ and $\si\in \R$.
The most interesting noise is $\xi= \dot{w}$ and, in this case, $\si_\xi=0$ holds.

Then, from $\eta$ or equivalently from $\xi$, we define a function $\th(x)=\th_\xi(x)$ 
on $\T$ and a constant $\mu=\mu_\xi \in \R$, respectively, by
\begin{equation} \label{eq:th+A1}
\begin{aligned}
&\th(x) := e^{-\eta(x)} \Big\{ \mu \int_0^xe^{\eta(y)}dy + 1\Big\}, 
\quad x \in \T,\\
& \mu:= \frac{e^{\eta(1)}-1}{\int_0^1e^{\eta(y)}dy}.
\end{aligned}
\end{equation}
Note that $\th(x)>0$ so it is uniformly positive by its continuity.  Indeed, this is
obvious if $\mu\ge 0$, while $\th(x) \ge e^{-\eta(x)} \{ \mu \int_0^1e^{\eta(y)}dy + 1\} = 
e^{-\eta(x)+\eta(1)}>0$ if $\mu<0$.  Moreover, $\th$ is periodic: $\th(0)=\th(1)=1$.
In case $\si=0$, we have $\mu=0$ (and vice versa) and $\th(x)=e^{-\eta(x)}$.
Note that $\th=\th(x)$ satisfies 
\begin{align} \label{th-mu}
\nabla\th+\xi\th=\mu 
\end{align}
at least if
$\xi\in C(\T)$, that is, $\eta\in C^1([0,1])$, and therefore $\De(z\th)+
\nabla(z\th\cdot\xi)=0$ holds for every $z\in \R$.  In particular, $v=\fa^{-1}(z\th)$
are stationary solutions of \eqref{eq:0-v}, where $\fa^{-1}$ is the inverse function of $\fa$.

For each conserved mass $m\in \R$ given as in \eqref{eq:mass},
determine $z=z_m\in \R$ uniquely by the relation
\begin{equation} \label{eq:m-z}
m= \int_\T \fa^{-1}(z\th(x))dx.
\end{equation}
Note that,  since $\fa$ satisfies \eqref{c-fa} and in particular, it is strictly increasing, this determines a
one to one relation between $z$ and $m$.
Then,
\begin{equation}  \label{eq:ST}
\bar v(x)\equiv \bar v_m(x) := \fa^{-1}(z_m\th(x))
\end{equation}
is a stationary solution of \eqref{eq:0-v} satisfying $\int_\T v dx=m$ in distributional sense,
or at least if $\xi\in C(\T)$.  Indeed, the mass conservation law is clear and,
as we noted above,
\begin{align*}
\De \{\fa(\bar v_m)\} + \nabla\{\fa(\bar v_m) \xi\} &= z_m \nabla \{\nabla \th+ \th\xi\}
 = z_m \nabla\mu=0,
\end{align*}
so that the right hand side of \eqref{eq:0-v} vanishes for $v=\bar v_m$.

As we will see, at least if $|\mu_\xi|$ is small enough, $\bar v_m$ is the unique 
stationary solution of the SPDE  \eqref{eq:0-v} for each fixed $m$, where 
$\mu_\xi$ is the constant determined in \eqref{eq:th+A1} from the noise $\xi$.  
In fact, we will show in Section \ref{sec:3.1} that, for initial value $v_0(x)$ satisfying \eqref{eq:mass}
(with $v(t)$ replaced by $v_0$), the solution $v(t,x)$ of the SPDE \eqref{eq:0-v}
converges to the stationary solution $\bar v_m$ as $t\to \infty$
at least if $|\mu_\xi|$ is small.
Moreover, without assuming the smallness of
$|\mu_\xi|$, the SPDE \eqref{eq:0-v} has a global solution in time.

\begin{thm} \label{thm:1.1}
Let $\fa \in C^4(\R)$ satisfy \eqref{c-fa} and  $\a\in (\frac{13}9,\frac32)$.
Then, for every initial value $v_0\in C^{\a-1}$,
the SPDE \eqref{eq:0-v} has a global-in-time
solution $v(t)\in C^{\a-1}$ for all $t\ge 0$.  Moreover, if $|\mu_\xi|$
is sufficiently small, $v(t)$ converges exponentially fast 
to $\bar v_m$ in $C^{\a-1}$ as $t\to\infty$:
\begin{align} \label{eq:exp-1.9}
\| v(t)-\bar v_m\|_{C^{\a-1}}\le C e^{-ct},
\end{align}
for some $c,C>0$, 
where $m$ is determined from $v_0$ as $m=\int_\T v_0(x)dx$.
\end{thm}

We apply the energy inequality, Poincar\'e inequality, the continuity of
the solutions in enhanced noise and  initial values, and also the initial layer property
to show Theorem \ref{thm:1.1}.  See Remark \ref{rem:2.2} below for the SPDE 
\eqref{eq:2} with general $\chi$ instead of \eqref{eq:0-v}.

This theorem for the slope $v(t) =\nabla u(t)$ of $u(t)$ implies
the following result for $u(t)$ itself.

\begin{thm} \label{thm:1.1-u(t)}
Assume $a=g'\in C^3(\R)$ (so that $g\in C^4(\R)$), the condition \eqref{eq:acC} and $\a\in (\frac{13}9,\frac32)$.
Then, the SPDE \eqref{eq:1.1} has a global-in-time solution
$u(t)\in C^\a$ for all $t\ge 0$.  Moreover, if $|\mu_\xi|$ is sufficiently small
as in Theorem \ref{thm:1.1}, $u(t)$ has the following uniform bound in $t$:
\begin{equation}  \label{eq:unif-u(t)}
\sup_{t\ge 0} \|u(t)-z_0\mu_\xi t\|_{C^\a}<\infty,
\end{equation}
where $z_0$ is defined by \eqref{eq:m-z} with $m=0$.
In particular, we have
$$
\lim_{t\to\infty} \frac1tu(t,x)=z_0\mu_\xi
$$
uniformly in $x\in\T$.
\end{thm}

\begin{rem}  \label{rem:1.1}
{\rm (i)} In Theorems \ref{thm:1.1} and \ref{thm:1.1-u(t)}, we assume $\a\in (\frac{13}9,\frac32)$, which is slightly more restrictive than the original assumption $\a\in (\frac43,\frac32)$ in \cite{FHSX}. This is because of Theorem \ref{thm:CIN}, which has been used for the proofs of Theorems \ref{thm:1.1} and \ref{thm:1.1-u(t)}, see Proposition \ref{prop:2.6-B} in Section \ref{sec:3.1}. 
For the reason for changing the range of $\a$, see Remark \ref{rem-3.2-R2111}.
Except this point, the other statements in  Section \ref{sec:3.1} hold for  $\a\in (\frac43,\frac32)$. So, unless  otherwise noted, we still assume $\a\in (\frac43,\frac32)$ throughout this article.
\\
{\rm (ii)} For the space-time white noise case as in \cite{FS}, the average (i.e. integral on
$\T$) of $u(t)$ behaves as a Brownian motion and it never converges as $t\to\infty$.
But, in our case, noise is only spatially dependent and the situation is
different.  Removing the constant drift, $u(t)$ stays bounded in $t$.
\end{rem}

We remark that the existence of global solutions of singular semilinear SPDEs are 
known, for example, for the following models. The linear equation \eqref{eq:1.1} 
with $a=1, g=v$ on $\R^d$ (i.e., the equation \eqref{eq:0-v} with $\fa(v)=v$) 
is studied in \cite{CC} and respectively, the generalized parabolic Anderson model
(PAM) (i.e., the equation \eqref{eq:0-v} with $\fa(v)=v$ and without $\nabla$)
in \cite{GIP-15} (Remark 5.4) and \cite{BB} by different approaches. 
For the nonlinear case, the dynamic $\phi_3^4$-model on $\T^3$ is studied 
in \cite{MW} by establishing a priori estimate, and the complex Ginzburg-Landau
equation on $\T^3$ in  \cite{Ho-18}.  The global existence for multi-component
coupled Kardar-Parisi-Zhang (KPZ) equation is shown in \cite{FH} under the
trilinear condition by studying its stationary measure. In addition, there are 
few works on the exponential decay in time of the solutions of singular semilinear SPDEs. 
For instance, \cite{TW} showed the exponential decay for the dynamic
$P(\phi)_2$-model on $\T^2$ to its unique invariant measure with respect 
to the total variation norm as $t\to\infty$,  and recently \cite{GP} showed 
the exponential $L^2$-ergodicity of conservative stochastic Burgers equation
on $\T$ based on the approach of the martingale problem.  Among these, 
to the best of our knowledge, our result is the first one in quasilinear case.

\subsection{Continuity of the local solution in initial values and parameter $m$}
\label{subsec:1.3}

Let  $u(t, \hat{\xi}, u_0), t\leq T$ denote the local-in-time paracontrolled solution of \eqref{eq:1.1} for each fixed  enhanced noise $\hat{\xi}:=(\xi, \Pi(\nabla X, \xi))$ and initial value $u_0$.  
In \cite{FHSX}, it is shown that the paracontrolled solution $u(t, \hat{\xi}, u_0)$ is continuous in the enhanced noise $\hat{\xi}$ for each fixed initial value $u_0$, see Theorem 3.1-{\rm (ii)} of \cite{FHSX}. 
In Section \ref{sec:3}, we show the joint continuity of  $u(t, \hat{\xi}, u_0)$ in $(\hat{\xi}, u_0)$ without the restriction $a=g'$.  More precisely, the main result of Section \ref{sec:3}
is the next theorem.
\begin{thm} \label{thm:CIN}
Let $\a \in (\frac{13}{9}, \frac{3}{2})$ and $u(t,\hat\xi, u_0)\in C^\a$ be the unique local paracontrolled solution of \eqref{eq:1.1} with initial value $u_0\in C^\a$ at least up to time 
$T=T(\|u_0\|_{C^\a}, \|\hat{\xi}\|_{C^{\a-2} \times C^{2\a -2}})$. Then, $u(t,\hat\xi, u_0)$ is continuous in $(t,\hat\xi, u_0)$ in the
region $\{(t,\hat\xi, u_0)\in [0,\infty)\times(C^{\a-2}\times C^{2\a-3})
\times C^\a\, ;\, t\leq T\}$. In particular, for each $0<t\leq  T$, the solution $u(t)$ of  \eqref{eq:1.1} is continuous in its initial values. 
\end{thm}
This theorem can be easily proved by Theorem \ref{thm:ContInintial} in Section \ref{sec:3},
see also Remark \ref{rem-3.2-R2111}.
Furthermore, the continuity in $m$ of the solution $u= u(t, m, \hat{\xi}, u_0)$
of \eqref{eq:1.1} with the coefficients $a$ and $g$ replaced by $a(\cdot+m)$ and 
$g(\cdot+m)$, respectively, can be shown by a straightforward extension of
our estimates, see Remark \ref{rem-3.3}.

\begin{rem} \label{rem-1.2-210521}
{\rm (i)}
For the special case $a=g'$, we have that  $u(t,\hat\xi, u_0)$ of \eqref{eq:1.1}  is continuous in $(t,\hat\xi, u_0)$ in the
region $\{(t,\hat\xi, u_0)\in [0,\infty)\times(C^{\a-2}\times C^{2\a-3})
\times C^\a \}$ by Theorems \ref{thm:1.1-u(t)}  and \ref{thm:CIN}.
\\
{\rm (ii)} 
From the explanation in Subsection \ref{sec:6.1}, we see that  Theorem \ref{thm:CIN} implies that the solution $v(t,\hat\xi, v_0)$ of  \eqref{eq:2} is continuous in the
region $\{(t,\hat\xi, v_0)\in [0,\infty)\times(C^{\a-2}\times C^{2\a-3})
\times C^{\a-1}; t \leq T \}$ and in particular, $v(t,\hat\xi, v_0)$ is continuous   in its initial values $v_0$  up to time 
$T$. 
Moreover, by Theorem \ref{thm:1.1}, for  the special type of SPDE \eqref{eq:0-v},  we know that its solution $v(t, v_0)$ is continuous in 
$(t, v_0)$ in the region $\{(t, v_0) \in [0, \infty) \times C^{\a-1} \}$. 
\end{rem}
\begin{rem} 
Hairer and Mattingly \cite{HM} discussed as follows: Let $U$ be
 the state space of solutions of certain SPDE 
and $\bar U := U \cup \{\De\}$ by adding the death point $\De$ to cover blow-up
of solutions.  They consider the solution of SPDE as a random dynamical
system $\Phi_{s,t}: \bar U\times M \to \bar U$, where $M$ denotes the space of admissible models for a given
regularity structure.  In our case, we may fix the noise (or more precisely, the enhanced noise), and consider the solution of SPDE as a deterministic dynamical system
($\Phi_{t}: \bar U \to \bar U$ would be enough).  Their Assumption 1 is 
for the continuity of the map $\Phi_{s,t}$.
They first prove the strong Feller property under general setting, 
and then show the assumptions formulated in general setting for several singular SPDEs.
In our case, once the noise is fixed, the solution is deterministic
so that what we need is the continuity of the solution in the initial value.
\end{rem}

\subsection{Definition of the solution of \eqref{eq:2} and relation to \eqref{eq:1.1}}
\label{sec:6.1}

Recall that the solution $u(t)$ of \eqref{eq:1.1} was defined in 
paracontrolled sense by solving the fixed point problem for the map $\Phi$ 
defined by (2.16) in the class $\mathcal{B}_T(\la)$,
see \cite{FHSX}.  

Let us first  remark the following lemma which generalizes Theorem 1.1 of  \cite{FHSX} and is shown by the cut-off argument.

\begin{lem}  \label{lem:1.4}
Let $\a \in (\frac43, \frac32)$, $a, g \in C^3(\R)$ and the condition \eqref{eq:acC} be satisfied. 
Then, the SPDE  \eqref{eq:1.1} has a unique local-in-time solution $u(t)$ 
defined in paracontrolled sense.
\end{lem} 
\begin{proof}
For a given $a\in C^3(\R)$, we can take  a sequence of functions 
$a^n\in C_b^3(\R)$ such that $0<c_- \leq a^n \leq c_+$, $a^n(v)=a(v), |v|\leq n$
and $a^n$ converges  uniformly to $a$ on each compact set.
Similarly, for a given $g\in C^3(\R)$,  let us take a sequence of functions 
$g^n\in C_0^3(\R)$ such that $g^n(v) =g(v), |v| \leq n$
and $g^n$ converges  uniformly to $g$ on each compact set.
Let us consider the equation
\begin{equation}  \label{eq:1.1-0}
\partial_t u^n = a^n(\nabla u^n)\De u^n + g^n(\nabla u^n)\cdot\xi
\end{equation}
starting from $u_0\in C^\a$.
Then, the assumptions of  Theorem 1.1 of \cite{FHSX} are satisfied and we know that \eqref{eq:1.1-0}  has a unique paracontrolled solution $u^n= u^n(t)$ up to a time $T_*^n>0 \ a.s.$, which is similarly defined by \eqref{eq:Tstar}. Without loss of generality, we may assume $T_*^n<\infty\ a.s.$ Let us now define $\tau^n$ by  
$\tau^n =\inf\{t>0: \|u^n(t)\|_{C^\a} \geq n \}$. Then, it is clear that $\tau^n >0\ a.s.,$ whenever $\|u_0\|_{C^\a} < n$  and Lemma \ref{lem:6.2}  below gives that $\tau^n \leq T_*^n$. Moreover, we have $u^m(t) =u^n(t), t\leq \tau^m \wedge \tau^n$.   Therefore, one knows that  \eqref{eq:1.1}  has a unique solution 
 $u(t)$ in the paracontrolled sense at least up to the  time  $\lim_{n \to \infty} \tau^n>0 \ a.s.$
\end{proof}

As we mentioned, \eqref{eq:2} is obtained at least for a smooth noise
from \eqref{eq:1.1} with proper modification in $m$ by differentiation.  
Motivated by this, we give the meaning to the equation \eqref{eq:2}
indirectly via the equation \eqref{eq:1.1}.  Let an initial value 
$v_0\in C^{\a-1}, \a\in (\frac43,\frac32)$ of  \eqref{eq:2} be given.
Then, set $m:= \int_\T v_0(x)dx$ and define $u_0\in C^\a$ by 
integrating $v_0-m$ as
\begin{equation} \label{eq:initial-v-u}
u_0(x) = \int_0^x (v_0(y)-m) dy +C , \quad x \in \R,
\end{equation}
for any constant $C\in \R$.
We solve \eqref{eq:1.1} with $a(v)=\fa'(v+m), g(v)=\chi(v+m)$ and
this initial value $u_0$ in paracontrolled sense.  The solution is denoted by
$u(t)=u(t;v_0,C)$.  Recall $u(t)\in \mathcal{L}_T^\a$ for some $T>0$.

\begin{defn}  \label{defn:sol-1.2}
We call $v(t) := \nabla u(t;v_0,C)+m\in \mathcal{L}_T^{\a-1}$ the solution
of the SPDE \eqref{eq:2} with initial value $v_0$. 
\end{defn}

Note that, if the noise $\xi$ (and $v_0$) is smooth, $v(t)$ is a smooth
classical solution of \eqref{eq:2}.  Indeed, in such case, $u(t)$ is a smooth
solution of \eqref{eq:1.1} so that this follows by differentiation.
Note also that $v(t;v_0,C)$ does not depend on the choice of $C$.
Indeed, again for smooth smeared noise $\xi^\e$, we easily see
$u^\e(t;v_0,C)=u^\e(t;v_0,0)+C$ for the corresponding solutions
$u^\e$ of \eqref{eq:1.1}.  Thus, by applying Theorem 1.1 of \cite{FHSX}
and taking the limit $\e\downarrow 0$, we see that $u(t;v_0,C)=u(t;v_0,0)+C$
holds for general noise $\xi$.  This implies $\nabla u(t;v_0,C)=
\nabla u(t;v_0,0)$.  In particular, $v(t)$ is well-defined.

Conversely, $u(t)$ can be recovered from $v(t)=\nabla u(t)$
(with $m=0$) as follows.  Assume $\xi\in C^\infty(\T)$
and let the initial value $u_0\in C^{\a}, \a\in (\frac43,\frac32)$
of  \eqref{eq:1.1} be given.  Then, we determine $v(t)$ by solving
\eqref{eq:2} with initial value $v_0:= \nabla u_0$, and set
$$
u(t,x):= \int_0^x v(t,y)dy+\int_\T u_0(y)dy -\int_\T (1-y)v(t,y)dy
+ \int_0^tds \int_\T \chi(v(s,y))\cdot\xi(y)dy.
$$
(In the right hand side, especially in the first and third terms, we regard $\T=[0,1)$.)
Then, one can show that $u(t)$ solves the equation \eqref{eq:1.1} with $a=\fa',
g=\chi$, see Lemma \ref{lem:2.Section2.3} below.
At least if $\xi$ is smooth, the equivalence between \eqref{eq:1.1}
and \eqref{eq:2} is established.

Moreover, concerning the renormalizations, the equation \eqref{eq:1.1} in
integrated form does not require them, since the resonant term
$\Pi(\nabla X,\xi)$ involves the derivative of $X$
 as we discussed in \cite{FHSX} (though \eqref{eq:1.1} is
 an analog of KPZ equation).  In particular, the equation
 \eqref{eq:2} in differentiated form does not require them too.

The remainder of this article is organized as follows.
Section \ref{sec:3.1} is for the proofs of Theorems \ref{thm:1.1} and \ref{thm:1.1-u(t)}.
We first formulate the energy estimate for \eqref{eq:0-v}  driven by a smooth noise 
in Subsection \ref{s:Ene:1D}, see Proposition \ref{prop:Phi-f}.  Then,
the proof of Theorem \ref{thm:1.1} is given in Subsection \ref{sec:2.3}. 
We note the continuity of the solution in the enhanced noise $\hat \xi$, the
initial values and the parameter $m$ in the coefficients.
We derive Poincar\'e inequality and show that the Poincar\'e constant can be 
taken uniformly in the approximating sequence of the noise.  We also rely on the 
initial layer type property of the solution of the SPDE \eqref{eq:0-v}, that is, 
the regularity of the solution is improved in an arbitrary short time.  
Subsection \ref{sec:2.3-u} is devoted to the proof of  Theorem \ref{thm:1.1-u(t)}
based on the relation between  \eqref{eq:1.1} and \eqref{eq:2}.  In Section \ref{sec:3}, 
we show Theorem \ref{thm:CIN} by establishing Theorem \ref{thm:ContInintial}.

\section{Global solvability and convergence to stationary solution}
\label{sec:3.1}

We show the global solvability of \eqref{eq:0-v} based on the energy inequality
and Poincar\'e inequality.
This gives the exponentially fast convergence in $C^{\a-1}$ of the solution
$v(t)$ to the stationary one, first for the initial value $v(0)\in \mathcal{D}$,
at least if $|\mu_\xi|$ is sufficiently small.
Here $\mathcal{D}$ is the class of all functions $v\in C^{\a-1}$ satisfying 
$\fa(v)\th^{-1}\in H^1$, where $H^1=H^1(\T)$ is the Sobolev space on $\T$.
Then, this result will be extended to general initial values
$v(0)\in C^{\a-1}$ by the initial layer type property of the solution.  
Note that $v(0)\in C^\b, \b\in (\frac13,\a-1)$, 
is equivalent to writing $v(0)\in C^{\a-1}$, $\a\in (\frac43,\frac32)$, 
by tuning in $\a$.  In addition, we will use $C$ to denote  a positive generic constant that may change from line to line in this section.

\subsection{Method of energy inequality}
\label{s:Ene:1D}

In this subsection, we assume $\xi\in C^\infty(\T)$ (or at least $\xi\in C^2(\T)$)
and consider a differentiable solution $v(t,x)$ of \eqref{eq:0-v}.
More precisely, if $\xi\in C^2(\T)$, the equation \eqref{eq:2} is a classical PDE of
divergence form:
\begin{align}  \label{eq:2.1-A}
\partial_t v= \nabla\{a_1(v,\nabla v)\} - a_2(x,v,\nabla v), \quad x \in \T,
\intertext{or, we can further rewrite it as}
\partial_t v= a_3(v)\De v- A(x,v,\nabla v),\quad x \in \T,
\label{eq:2.2-A}
\end{align}
where $a_1(v,p)= \fa'(v)p$, $a_2(x,v,p)=-\big(\dot{\xi}(x)\chi(v)+\xi(x)\chi'(v)p\big)$,
$a_3(v)=\fa'(v)$ and $A(x,v,p) = -\fa''(v)p^2+ a_2(x,v,p)$.  Note that 
\eqref{eq:2.1-A} and \eqref{eq:2.2-A} are written in the forms of (6.1) and (6.4)
in \cite{LSU} (p.449, p.450), respectively.  We actually consider \eqref{eq:0-v}
so that $\chi=\fa$.  Recall that $\fa\in C^4(\R)$ satisfies \eqref{c-fa} and
this, in particular, implies the linear growth property of $\fa$: $|\fa(v)|
\le C(|v|+1)$.  Thus, we see that the conditions a)--d) of Theorem 6.1
(\cite{LSU}, p.452), especially,
\begin{align*}
&\partial_pa_1(v,p) \,(=\fa'(v)) \ge 0, \quad A(x,v,0)v \ge -b_1v^2-b_2, \\
& 0<c_-\le \partial_p a_1(v,p)\le c_+, \quad
(|a_1|+|\partial_v a_1|)(1+|p|) +|a_2| \le C_M(p^2+1)
\end{align*}
hold for $x \in \T, |v|\leq M, p\in\R$.  Indeed, the second bound follows from
\begin{align*}
|A(x,v,0)v|=|a_2(x,v,0)v|=|\dot{\xi}(x)| \, |\fa(v)v|
\le C_2(v^2+1).
\end{align*}
Note that $M>0$ given in (6.8) of \cite{LSU} can be taken in our situation
by applying the maximum principle, see Remark \ref{rem:2.1}-{\rm (i)} below.  The condition
$\xi\in C^2(\T)$ is required for the condition c).  The condition d) is shown by
the boundedness of $\partial_v a_1$, $\partial_p a_2$, $\partial_v a_2$ for
$|v|\le M$ and $|p|\le M_1$ for each $M, M_1>0$.  Therefore, by Theorem 6.1 
of \cite{LSU}, \eqref{eq:2} has a unique global-in-time classical solution
$v(t,x) \, (\in H^{1+\b/2,2+\b}([0,T]\times \T))$ 
if $\xi\in C^2(\T)$ and the initial value $v(0)\in C^{2+\b}(\T)$ 
for some $\b\in (0,1)$.  Furthermore,
noting that 
$|a_1(v, p)| + |\partial_v a_1(v, p)| \leq C_M |p|, |v|\leq M, p\in \R$, 
by Theorem 6.4 (\cite{LSU}, p.460), the existence
of global-in-time classical solution is known if $\xi\in C^2(\T)$ and 
$v(0)\in C^{\b}(\T)$, $\b\in (0,1)$.  Note that, from examples in p.99 of \cite{BCD} 
or p.62 -L. 8 of \cite{GIP-15}, we know that $C^{\a} =H^{\a}$ (H\"older space 
used in \cite{LSU}) for all $\a\in \R^+ \setminus \N$.
One can easily check that the classical solution of \eqref{eq:2} is
a solution in paracontrolled sense.  This shows that
 the life time of the solution $v(t) \in C^{\a -1}$ of \eqref{eq:0-v} equals to infinity, i.e.,
\begin{align*}
T_*\equiv T_*(\hat \xi, v(0)):= \sup\{t \geq 0; \ \text{solution of } \eqref{eq:0-v}\  \text{with initial value} \  v(0) \ \text{exists} \} =\infty,
\end{align*}
if $\xi\in C^2(\T)$ and $v(0)\in C^{\a-1}$, $\a>1$.
We expect that $T_*$ is lower semicontinuous in  $\hat\xi$ as in \cite{FH},
but this combined with $T_*=\infty$ for $\xi\in C^2(\T)$ does not imply the same for
general $\xi\in C^{\a-2}$.

As we mentioned, in this subsection, we assume $\xi\in C^\infty(\T)$ and $v(t,x)$ is a smooth
global-in-time solution of \eqref{eq:0-v}, i.e., $v\in C^{1,2}((0,\infty)\times \T)
\cap C([0,\infty)\times \T)$.  We define $f(t,x)$ as
\begin{equation} \label{eq:f}
f(t,x):= \frac{\fa(v(t,x))}{\th(x)},
\end{equation}
where $\th=\th_\xi(x)$ is defined in \eqref{eq:th+A1}. Note that
$f(t,x)$ is periodic in $x\in \T$.  Then, we have
\begin{align*}
\nabla\fa(v) &= \nabla(f\th) = \nabla f \cdot \th + f \nabla \th \\
&= \nabla f \cdot \th + f (-\xi \th +\mu)\\
& = \nabla f \cdot \th + \mu f - \xi \fa(v).
\end{align*}
Therefore, the equation \eqref{eq:0-v} can be rewritten as
\begin{equation}  \label{eq:3.1}
\partial_t v =  \nabla(\th\nabla f + \mu f).
\end{equation}

Let $L^2_\th:= L^2(\T,\th dx)$ and $H^1_\th:= H^1(\T,\th dx)$ be the spaces
equipped with the norms $\|f\|_{L_\th^2} := \big(\int_\T f^2 \th dx\big)^{1/2}$
and $\|f\|_{H_\th^1} := \big(\int_\T \{f^2+(\nabla f)^2\}\th dx\big)^{1/2}$,
respectively.  We define the functional $\Phi(f)\equiv \Phi_\th(f)$ of $f \in H_\th^1$ as
$$
\Phi(f) \equiv \Phi_\th(f):=\frac12\int_\T (\nabla f)^2 \th dx.
$$

\begin{lem} \label{lem:3.1}
The functional $\Phi$ is Fr\'echet differentiable in $H_\th^1$ and its Fr\'echet
derivative $D\Phi(f)\in (H_\th^1)^*$ is given by
\begin{equation} \label{eq:Dphi}
D\Phi(f)(\psi) \equiv {}_{(H_\th^1)^*}\lan D\Phi(f),\psi\ran_{H_\th^1} =
\int_\T \nabla f\nabla \psi \, \th dx,
\end{equation}
for $\psi\in H_\th^1$. 
If $f\in C^2(\T)$, this is further rewritten as
$$
\int_\T D\Phi(x,f) \psi(x) \, \th dx,
$$
with
\begin{equation} \label{eq:Dphi-x}
D\Phi(x,f) = - \th^{-1} \nabla(\th \nabla f),
\end{equation}
note that $\th^{-1}$ means $\frac{1}{\th}$.
\end{lem}

\begin{proof}
Take $\psi\in H_\th^1$ and define $D\Phi(f)(\psi)$ as \eqref{eq:Dphi}.
Then,
$$
\Phi(f+\psi)-\Phi(f)-D\Phi(f)(\psi) = \Phi(\psi) = o(\|\psi\|_{H_\th^1})
$$
as $\|\psi\|_{H_\th^1} \to 0$. This shows the Fr\'echet differentiability of
$\Phi$  and the formula \eqref{eq:Dphi}.  The formula \eqref{eq:Dphi-x} 
for $D\Phi(x,f)$ is
shown by a simple integration by parts:
\begin{align*}
\int_\T \nabla f\nabla \psi \, \th dx
 = -\int_\T \th^{-1} \nabla(\th \nabla f) \psi \, \th dx.
\end{align*}
\end{proof}

Noting $v= \fa^{-1}(f\th)$ from \eqref{eq:f} and using \eqref{eq:Dphi-x}, \eqref{eq:3.1} is rewritten as
\begin{equation}  \label{eq:3.2}
\partial_t (\fa^{-1}(f\th)) =  -\th D\Phi(x,f) + \mu\nabla f.
\end{equation}
Set $G(x,f) = (\fa^{-1})'(f \th(x)) >0$ and 
\begin{align*}
K(x,f) = \frac1{G(x,f)} = \fa'(\fa^{-1}(f\th)) \, (=\fa'(v)) \ge c_- >0,
\end{align*}
recall the assumption \eqref{c-fa}.
Then, since $\partial_t\fa^{-1}(f\th) = G(x,f)\partial_t f \cdot \th$,
\eqref{eq:3.2} can be further rewritten as
\begin{equation}  \label{eq:3.3}
\partial_t f = K(x,f) (-D\Phi(x,f) + \mu \th^{-1}\nabla f),
\end{equation}
which is sometimes called Onsager equation at least when $\mu=0$,
see \cite{Mi}, p.193.  See also Remark \ref{rem:2.1} below for this equation.

\begin{prop} \label{prop:Phi-f}
Assume $\xi\in C^\infty(\T)$ and $v(t,x)$ is a smooth global-in-time solution  
of \eqref{eq:0-v}.  Then, for $f(t)$ defined by \eqref{eq:f}, if $f(0)\in H_\th^1$, we have
the bound
\begin{equation}  \label{ea:apriori-0}
\Phi(f(t)) \le \Phi(f(0))e^{C(\th) t},
\end{equation}
where $\th=\th_\xi$,
$$
C(\th) = - \frac{c_-}{2c_2(\th)} + \frac1{2c_-}\mu^2 c_1(\th)^2,
$$
$c_1(\th)= c_1(\min\th)$ defined by \eqref{eq:2.12-c1}  and $c_2(\th)>0$ is the constant
given in \eqref{eq:Cw} in Poincar\'e inequality.  (Note that $c_2(\th)$ stays finite
for every $\eta\in C([0,1])$.)
In particular, if $|\mu|=|\mu_{\xi}|$ is small enough, $C(\th)<0$ and this shows
the exponential decay of $\Phi(f(t))$ as $t\to\infty$:
\begin{equation}  \label{ea:apriori}
\Phi(f(t)) \le \Phi(f(0))e^{-c_* t},
\end{equation}
for some $c_*>0$.  When $\mu_{\xi}=0$, in particular, when $\si=\lan \xi, 1\ran =0$, one can take
$c_*= \frac{c_-}{c_2(\th)}$ (better than $C(\th)$ with $\mu_{\xi}=0$).
\end{prop}

\begin{proof}
Recalling $\th(x) > 0$ and $K(x,f) =\fa'(v)\ge c_->0$, we obtain from \eqref{eq:3.3}
\begin{align}  \label{eq:2.9-0}
\partial_t \Phi(f) & = \lan \partial_t f, D\Phi(\cdot,f)\ran_{L_\th^2} \\
& = -\int_\T K(x,f) D\Phi(x,f)^2 \th dx + \mu \int_\T K(x,f)D\Phi(x,f)\nabla f dx 
 \notag \\
& \le -c_- \|D\Phi(\cdot,f)\|_{L_\th^2}^2 + \mu \int_\T K(x,f)D\Phi(x,f)\nabla f dx.
\notag
\end{align}
For the second term, since $\fa$ satisfies \eqref{c-fa}, we have $c_-\le K(x,f)\le c_+$
and this shows
$$
\le |\mu| \, c_+ \| D\Phi(\cdot,f)\|_{L^2_\th} \| \nabla f \|_{L_{\th^{-1}}^2}.
$$
However, since $\th$ is uniformly positive, $\th^{-1}(x) \le c(\th) \th(x)$ for 
$c(\th):= (\min_{x\in\T}\th^2(x))^{-1}>0$ and therefore 
$\| \nabla f \|_{L_{\th^{-1}}^2}\ \le \sqrt{c(\th)} \| \nabla f \|_{L_\th^2}
= \sqrt{2 c(\th)} \Phi(f)^{\frac12}$.  Thus, the second term is bounded by
\begin{align*}
&\le |\mu| \, c_1(\th) \| D\Phi(\cdot,f)\|_{L^2_\th} \Phi(f)^{\frac12}\\
& \le \tfrac{c_-}2  \| D\Phi(\cdot,f)\|_{L^2_\th}^2 + \tfrac1{2c_-} \mu^2 c_1(\th)^2 \Phi(f),
\end{align*}
where 
\begin{align} \label{eq:2.12-c1}
c_1(\th)\equiv c_1(\min\th):= c_+\sqrt{2c(\th)} >0.
\end{align}  Therefore, we obtain
\begin{align}\label{eq:3.4-0}
\partial_t \Phi(f) \le
-\tfrac{c_-}2  \| D\Phi(\cdot,f)\|_{L^2_\th}^2 + \tfrac1{2c_-} \mu^2 c_1(\th)^2 \Phi(f).
\end{align}
We now apply Poincar\'e inequality $\Phi(f)\le c_2(\th)\|D\Phi(\cdot,f)\|_{L^2_\th}^2$ 
given in Lemma \ref{lem:4.1} below, and then \eqref{eq:3.4-0} shows that
\begin{align*}
\partial_t \Phi(f) \le
-\tfrac{c_-}{2c_2(\th)} \Phi(f) + \tfrac1{2c_-} \mu^2 c_1(\th)^2 \Phi(f)
= C(\th) \Phi(f).
\end{align*}
This implies $\partial_t \big( e^{-C(\th) t} \Phi(f)\big)\le 0$ and leads to
the bound \eqref{ea:apriori-0}.  \eqref{ea:apriori} is immediate from \eqref{ea:apriori-0}.
When $\mu_{\xi}=0$, 
$\partial_t \Phi(f) \le -c_-\| D\Phi(\cdot,f)\|_{L^2_\th}^2$ holds by \eqref{eq:2.9-0}, which is simpler than \eqref{eq:3.4-0}.
Therefore, one can take $c_*= \frac{c_-}{c_2(\th)}$ in this case by Lemma \ref{lem:4.1}.
\end{proof}

The following is the Poincar\'e inequality used in the proof of Proposition
\ref{prop:Phi-f}.

\begin{lem} \label{lem:4.1}
For every $f\in C^2(\T)$, we have
\begin{equation}  \label{eq:Poincare}
\Phi(f)\le c_2(\th) \|D\Phi(\cdot,f)\|_{L^2_\th}^2,
\end{equation}
where
\begin{equation}  \label{eq:Cw}
c_2(\th) := \frac12 \int_\T \th^{-1}(x)dx \int_\T \th(y)dy.
\end{equation}
\end{lem}

\begin{proof}
Set $g:= \th\nabla f$ and note that
$$
\int_\T g \th^{-1} dx = \int_\T \nabla f dx =0
$$
holds by the periodicity of $f$.  Then, noting that
\begin{align*}
& \Phi(f) = \frac12 \int_\T (g \th^{-1})^2 \th dx = \frac12 \int_\T g^2 \th^{-1} dx, \\
& \|D\Phi(\cdot,f)\|_{L^2_\th}^2 = \int_\T (\nabla g)^2 \th^{-1} dx,
\end{align*}
and setting
$$
Z := \int_\T \th^{-1} dx,
$$
we have
\begin{align*}
\Phi(f) &= \frac12 \int_\T \Big( g(x) - \frac1Z \int_\T g(y)\th^{-1}(y)dy \Big)^2 \th^{-1}(x) dx \\
&= \frac12 \int_\T \Big( \int_\T \big(g(x) -g(y)\big)  \frac1Z \th^{-1}(y)dy \Big)^2 \th^{-1}(x) dx \\
&\le \frac12 \int_\T \th^{-1}(x) dx \int_\T \big(g(x) -g(y)\big)^2  \frac1Z \th^{-1}(y)dy   \\
&= \frac1{2Z} \int_\T \th^{-1}(x) dx \int_\T \th^{-1}(y)dy \Big(\int_x^y \nabla g(z) dz\Big)^2   \\
&\le \frac1{2Z} \int_\T \th^{-1}(x) dx \int_\T \th^{-1}(y)dy \int_\T (\nabla g(z))^2 \th^{-1}(z) dz \int_\T \th(z)dz   \\
& = c_2(\th) \|D\Phi(\cdot,f)\|_{L^2_\th}^2,
\end{align*}
where we have used Schwarz's inequality twice.
This shows the conclusion.
\end{proof}

\begin{rem}  \label{rem:2.1}
{\rm (i)} The equation \eqref{eq:3.3} is rewritten as
\begin{equation}  \label{eq:3.3-B}
\partial_t f = K(x,f)\th^{-1} \nabla(\th\nabla f) + \mu K(x,f)\th^{-1}\nabla f.
\end{equation}
As we saw above, under the assumption $\xi\in C^\infty(\T)$, $f$ exists
globally in time and \eqref{eq:3.3-B} can be considered as a linear PDE for $f$
regarding the coefficient $K(x,f)$ is given.  Then, since the right hand side of
\eqref{eq:3.3-B} has no zeroth-order term in $f$, it satisfies the maximum
principle and we have
$$
\min_{x\in\T}f(0,x)\le f(t,x)\le \max_{x\in \T}f(0,x),
$$
see, for example, \cite{Evans} p.368.  Based on this observation and
taking the limit in $\xi$, one can cover the case $\th\in C(\T)$ and show
the global-in-time solvability of \eqref{eq:0-v}.  This provides another proof of the
first part of Theorem \ref{thm:1.1}, though the present article relies on 
the method of energy inequality.\\
{\rm (ii)}  (Linear case) When $\fa(v)=v$, we have $G(x,f)=K(x,f)=1$.
In addition, if $\mu_\xi=0$, or equivalently, if $\si_\xi=0$ holds,
\eqref{eq:3.3} defines a simple gradient flow:
\begin{equation}  \label{eq:3.5}
\partial_t f = -D\Phi(x,f),
\end{equation}
and this implies $\partial_t \Phi(f) = -\|D\Phi(\cdot,f)\|_{L^2_\th}^2$.
\end{rem}

\begin{rem}  \label{rem:2.2}
For the SPDE \eqref{eq:2} with general $\chi$ and smooth $\xi$, 
the stationary solution is a periodic solution
$v=v(x)$ of the ordinary differential equation
$$
\De\{\fa(v)\}+\nabla\{\chi(v)\xi\}=0.
$$
As before, setting $\th=\fa(v)$, this equation is rewritten as
\begin{equation}  \label{eq:rem-2.2-1}
\nabla\th + \psi(\th) \xi = \mu,
\end{equation}
where $\psi(\th):= \chi(\fa^{-1}(\th))$ and $\mu$ is any constant.
If $\mu=0$,  the equation \eqref{eq:rem-2.2-1} is of separable type
and solved as
$$
\Psi(\th) \equiv \int_0^\th \frac{d\th'}{\psi(\th')} = -\eta(x)+C, \quad x\in \T.
$$
For simplicity, if $\chi>0$, then $\Psi$ is increasing and \eqref{eq:rem-2.2-1} is
solved as
$$
\th_C(x)= \Psi^{-1}(-\eta(x)+C).
$$
For $\th_C$ to be periodic, $\eta$ should satisfy $\eta(0)=\eta(1)$, that is, $\si=0$.
In other words, the condition $\mu=0$ implies $\si=0$.  On the other hand,
the constant $C=C_m$ is determined from the conservation law \eqref{eq:mass}.

Once stationary solutions are found, to link them to the SPDE \eqref{eq:2},
we need to find a proper transformation like \eqref{eq:f} from $v$ to $f$, 
which extracts the factor $z$, that is $C$ in the present setting  for 
general $\chi$, and also a proper functional $\Phi(f)$ of $f$.  
However, this looks nontrivial.

Note that, in case $\chi=\fa$, $\psi(\th)=\th$, $\Psi(\th)=\log |\th|$ and
$\th_C(x) = \pm e^{-\eta(x)+C} = ze^{-\eta(x)}$, assuming $\mu=0$.
\end{rem}

\subsection{Proof of Theorem \ref{thm:1.1}}
\label{sec:2.3}

Now we consider general $\xi\in C^{\a-2}$, $\a\in (\frac43,\frac32)$.
We are discussing $v(t)$, but here start with $u(t)$, i.e.,   
 the unique local-in-time paracontrolled solution of \eqref{eq:1.1} with the 
 initial value $u_0\in C^\a$. 
 
Let us recall Theorem 3.1 {\rm (i)} in  \cite{FHSX}.  In that theorem, it is   declared that the map $\Phi$ defined by (2.16) in \cite{FHSX} (or see \eqref{Phi-6.5} below, which is 
defined in a little different setting from the original one) is contractive from 
$\mathcal{B}_T(\la)$ (a variant of \eqref{3.4-R21026} below)
into itself  for some large enough $\la$ and small enough $T>0$. But, the explicit choices
of $\la$ and $T$ were not given.  To show Lemma \ref{lem:6.2} below,
let us explicitly choose $\lambda$ and $T$.  They can be constructed easily 
by  the estimates  (3.48) and (3.50) obtained in the proof of Theorem 3.1 {\rm (i)} in \cite{FHSX}.  
In fact,  by these estimates, 
we know that there exists a large enough constant $M>0$ such that for all 
${\bf u}:=(u, u') \in \mathcal{B}_T(\la)$,
\begin{align}\label{eq-2.18-0510}
\|\Phi({\bf u})\|_{\a, \b, \ga} \leq M \Big(
 T^{\frac{\a + \b -\ga}2} K(\|{\bf u}\|_{\a, \b, \ga})\tilde{K}_1(X, \xi) 
 +K_0(\|u_0\|_{C^\a}) (1+\|\xi\|_{C^{\a -2}}) \Big), 
\end{align}
and 
\begin{align}\label{eq-2.19-0510}
\| \Phi( {\bf u}_1) - \Phi({\bf u}_2)\|_{\a, \b, \ga}
\leq  M T^{\frac{\a + \b -\ga}2} 
K(\|{\bf u}_1\|_{\a, \b, \ga}, \|{\bf u}_2\|_{\a, \b, \ga}) 
\|{\bf u}_1 -{\bf u}_2\|_{\a, \b, \ga}\tilde{K}_2(X,\xi),
\end{align}
where $\b\in (\frac13,\a-1)$, $\ga\in (2\b+1,\a+\b)$, $K(\la), K(\la, \la)$ denote the increasing and positive  functions in $\la>0$ introduced at the end of Section 2 in  \cite{FHSX},  
$\tilde{K}_1(X, \xi)$ and $\tilde{K}_2(X, \xi)$ are the positive  polynomial functions  used in (3.48) and (3.50) in \cite{FHSX}.  
Let us determine $\la$ and $T>0$ as follows.
\begin{align} \label{eq-2.12-210510}
\la =& 2 M \Big(\tilde{K}_1(X, \xi) +K_0(\|u_0\|_{C^\a})(1+\|\xi\|_{C^{\a-2}}) \Big),   \\
T =&  \min\left\{\Big(K(\la) + M K(\la, \la)\tilde{K}_2(X, \xi) \Big)^{-\frac{2}{\a +\b -\ga}},1 \right\},   \label{eq-2.12-R21120}
\end{align}
where $M$ is same as in \eqref{eq-2.18-0510}. 
Then, Theorem 3.1 {\rm (i)} of \cite{FHSX} can be restated as follows.

\begin{thm}[Theorem 3.1 {\rm (i)} of \cite{FHSX}] \label{thm-2.4-R21120}
Let  $\la$ and $T$ be defined by \eqref{eq-2.12-210510} and   \eqref{eq-2.12-R21120},  respectively.
Then, for any $u_0 \in C^\a$ (or equivalently $u'_0 \in C^\b$ and $u_0^\sharp \in C^\a$),  $\Phi$  is contractive from $\mathcal{B}_T(\la)$ into itself. 
In particular, $\Phi$ has a unique fixed point
on $[0,T]$, which is the unique solution  of the system (2.17) and (2.18) in \cite{FHSX} and it solves 
the SPDE \eqref{eq:1.1} on $[0,T]$ in the paracontrolled  sense. 
\end{thm}
\begin{proof}
It is enough to show  $\Phi$ is contractive from $\mathcal{B}_T(\la)$ into itself. By  \eqref{eq-2.18-0510} and  \eqref{eq-2.12-210510}, we easily have
 \begin{align*}
\|\Phi({\bf u})\|_{\a, \b, \ga} \leq \frac{\la}2 \Big(
 T^{\frac{\a + \b -\ga}2} K(\|{\bf u}\|_{\a, \b, \ga}) +1\Big).
\end{align*}
Then, noting that $T$ is given as \eqref{eq-2.12-R21120}, we have 
$T^{\frac{\a + \b -\ga}2} K(\|{\bf u}\|_{\a, \b, \ga}) <1$ whenever $\|{\bf u}\|_{\a, \b, \ga} \leq \la$ and therefore, $\Phi$ maps $\mathcal{B}_T(\la)$ into itself.
The contractivity of the map $\Phi$ on $\mathcal{B}_T(\la)$ is obvious by  \eqref{eq-2.19-0510} and  
$T^{\frac{\a + \b -\ga}2} M K(\la, \la) \tilde{K}_2(X, \xi) <1$.
\end{proof}

\begin{rem} \label{rem-2.2-R21115}
From Theorem \ref{thm-2.4-R21120},  
we see that the time $T$ chosen as in \eqref{eq-2.12-R21120} depends continuously on the norm 
$\|u_0\|_{C^\a}$(or equivalently on $\|u'_0 \|_{C^\b}$ and $\|u_0^\sharp\|_{C^\a}$), which 
is vital for the proof of the next Lemma \ref{lem:6.2}.
 \end{rem}

Let us define the explosion time $T_*$ by
\begin{equation}  \label{eq:Tstar}
T_* \equiv T_*(\hat \xi, u_0) := \sup\{t\ge 0\, ; \,\text{solution } u(t) \in C^\a \text{ of } \eqref{eq:1.1} \ \text{starting from}\   u_0 \  \text{exists} \}.
\end{equation}
We know that $T_*>0$.
Furthermore, we have the following result.
\begin{lem} \label{lem:6.2} 
If $T_*<\infty$, we have
$$
\lim_{t\uparrow T_*} \|u(t)\|_{C^\a}=\infty.
$$
\end{lem}
\begin{proof}

By the definition of $T_*$, the solution exists and satisfies $u(\cdot)\in C([0,T_*),C^\a)$.
If the conclusion does not hold, one can find $M>0$ and a sequence $t_n \uparrow T_*$
such that $\|u(t_n)\|_{C^\a} \le M$.  
However, by Theorem \ref{thm-2.4-R21120} and Remark \ref{rem-2.2-R21115}, 
there exists $\e=\e_M>0$, which is uniform in $n$, such that one can solve 
\eqref{eq:1.1} starting from 
$u(t_n)$ on the time interval $[t_n, t_{n} +\e]$. Since $t_n \uparrow T_*$, this 
shows that one can solve  \eqref{eq:1.1} beyond $T_*$ and contradicts the 
definition of $T_*$.
\end{proof}

\begin{rem} \label{rem-2.4-0}
In particular, let $C^\a\cup\{\De\}$ be the one point compactification of
$C^\a$ and define $u(t):=\De$ for $t\ge T_*$.  Then, $u(t)$ is defined for
all $t\ge 0$ and $u(\cdot)\in C([0,\infty), C^\a\cup\{\De\})$ by Lemma \ref{lem:6.2}.
Denoting $u(t)$ with initial value $u_0\in C^\a\cup\{\De\}$ by
$u(t,u_0)$, it has the flow property: $u(t,u(s,u_0))=u(t+s,u_0)$
for all $t,s\ge 0$.
\end{rem}

We extend the result of Proposition \ref{prop:Phi-f} to general noise
$\xi\in C^{\a-2}$, $\a \in (\frac{13}{9}, \frac{3}2)$.  Recall that $\mathcal{D}$ 
is the class of all 
functions $v\in C^{\a-1}$ satisfying $\fa(v)\th^{-1}\in H^1$.

\begin{prop} \label{prop:2.6-B}
Assume that the initial value of the SPDE \eqref{eq:0-v}
satisfies $v(0)\in \mathcal{D}$.  Then, the solution $v(t)$ exists
globally in time for all $t\ge 0$ and $f(t)$ defined from $v(t)$ by 
\eqref{eq:f} satisfies 
\begin{equation}  \label{eq:2.20}
\Phi(f(t))\le e^{Ct}\Phi(f(0)),
\end{equation}
for some $C\in \R$.
In particular, if $|\mu_\xi|$ is small enough, one can take $C<0$.
\end{prop}

\begin{proof}
First assume that $\xi \in C^\infty(\T)$ and let $f(t,x)$ be as in \eqref{eq:f}
and Proposition \ref{prop:Phi-f} defined from $v(t,x):=\nabla u(t,x)+m$,
where $m=\int_\T v(0,x)dx$ and $u(t,x)$ is a (classical) global-in-time
solution  of \eqref{eq:1.1} with initial value $u_0$ defined by \eqref{eq:initial-v-u}, $a=\fa'(\cdot+m)$ and $g=\fa(\cdot+m)$.
Then, we have the estimate \eqref{ea:apriori-0} for $\Phi(f(t))$,
if $f(0)\in H_\th^1$.

Now, let $\xi \in C^{\a-2}, \a\in (\frac{13}9,\frac32)$
be given and take a sequence of $\xi_n \in C^\infty(\T)$
such that $\hat\xi_n = (\xi_n, \Pi(\nabla X_n,\xi_n))$ converges to $\hat\xi=
(\xi, \Pi(\nabla X,\xi))$ in $C^{\a-2}\times C^{2\a-3}$ as $n\to\infty$, see Lemma 5.2 in \cite{FHSX} for details. 
Let $v_n(0):=\fa^{-1} (f(0) \theta_n)$, $m_n=\int_\T v_n(0,x)dx$, $a_n =\fa'(\cdot+m_n)$ and $g_n=\fa(\cdot+m_n)$, where $\th_n=\th_{\xi_n}$ is defined as in
\eqref{eq:th+A1} from the integral $\eta_n$ of $\xi_n$. 
Note that $v_n(0)$ is chosen in such a way that $f_n(0) (:=\fa(v_n(0))/\th_n) = f(0)$
holds foe every $n$.

We consider the SPDE  \eqref{eq:0-v} with the initial value $v_n(0)$ and the SPDE \eqref{eq:1.1} with the initial value $u_n(0)$ by replacing $a, g$ by $a_n, g_n$,  associated with $\xi_n$  respectively, where $u_n(0)$ is determined similarly to $u(0)$ above, see \eqref{eq:initial-v-u}.
Then, we have smooth classical global-in-time solutions $v_n, u_n$ for such 
equations.  Let $u$ and $v$ denote the solutions of \eqref{eq:1.1} and  
\eqref{eq:0-v} in paracontrolled sense associated with $\hat \xi$ for 
$t<T_*$.  Noting that  $\th_n  \to \th$ in $C^{\a -1}$ and  using assumptions 
on $\fa'$, we have, as $n \to \infty$, $v_n(0)$ converges to $v(0)$ in 
$C^{\a-1}$, $m_n \to m$ and in particular,  $a_n, g_n$ converge to $a, g$ on 
each compact set of $\R$. For the proof of $\th_n  \to \th$ in $C^{\a -1}$, we refer to the proof of Corollary \ref{prop:Global}, where a more complex case is dealt with.
Then, noting that $u_n(0) \to u(0)$ in $C^{\a}$ and using Remark \ref{rem-3.3} at the end of this article,  we know  that $u_n$ converges to $u$ in $\mathcal{L}_T^\a$
and therefore $v_n$ to $v$ in $\mathcal{L}_T^{\a-1}$, $\a-1<\frac12$ for $T<T_*$. Since the initial values $v_n(0)$ move, we use the continuity of solutions in initial values. 
Note that the coefficients $a_n$ and $g_n$ also move.
So, we require the condition $\a \in (\frac{13}{9}, \frac{3}2)$, see Theorem \ref{thm:CIN} or Remark \ref{rem-1.2-210521} for explanation, and use Remark \ref{rem-3.3} in the proof.

In particular, recalling $\eta(x)=\lan \xi, 1_{[0,x]}\ran \in C^{\a-1}([0,1])$ 
so that $\th_\xi\in C^{\a-1}(\T)$, $\th_\xi>0$  and also $\fa\in C^4(\R)$
satisfying \eqref{c-fa}, 
$f_n(t) := \fa(v_n(t)) \th^{-1}_n$ converges to $f(t):= \fa(v(t))\th^{-1}$ in 
$\mathcal{L}_T^{\a-1}$ and  $f_n(0)= f(0)$ for all $n$.  
  We note the lower semi-continuity of $\Phi(f)$
in $C^{\a-1}$, which follows from the variational formula for the Dirichlet form:
$$
\Phi(f) \equiv \Phi_\th(f) =\frac12\int_\T(\nabla f)^2\th dx
= \frac12 \sup\Big\{- \int_\T\frac{\nabla(\th \nabla w)}w f^2 dx; w\in C^2(\T)\Big\},
$$
when $\th\in C^1(\T)$, where $w(x)=0$ can happen only at $x$ such that $f(x)=0$.
Indeed, we may use the integration by parts and note 
$(\nabla f)^2-\nabla w \nabla(\frac{f^2}w)
= (\nabla f- f \frac{\nabla w}w)^2\ge 0$.
Then, $f_n\to f$ in $C^{\a-1}$ implies $\Phi_\th(f)\le \varliminf_{n\to\infty} \Phi_\th(f_n)$
if $\th\in C^1(\T)$, but in the definition of
$\Phi_\th(f)$, we only have $\th$ without its derivative so that this property 
holds also for non-smooth $\th$ by taking the limit $\th_m\to\th$  (in $L^\infty$) or
$\eta_m\to \eta$ (in $L^\infty$) introduced as above.
More precisely, noting that the constant $C(\th)$ in \eqref{ea:apriori-0}
can be taken uniformly in $n$: $C(\th_n)\le C$, since the constant $C(\th)=C(\th_\xi)$
can be estimated only by $\min\th$ and $\max\th$, for $t<T_*$, we see
\begin{align*}
\Phi_\th(f(t)) & = \lim_{m\to\infty} \Phi_{\th_m}(f(t))
\le \lim_{m\to\infty} 
\varliminf_{n\to\infty} \Phi_{\th_m}(f_n(t)) 
\le e^{Ct}\Phi_\th(f(0)),
\end{align*}
at least if $f(0)\in H_\th^1$; recall that  $f_n(0)=f(0)$ for all $n$ and $\th_n \to \th$ in $C^{\a -1}$.
For the last inequality, we use $\Phi_{\th_n}(f_n(t)) \le e^{Ct}\Phi_{\th_n}(f_n(0))$
by noting that, for arbitrary small $\e>0$, $\th_m/\th_n, \th_n/\th \le 1+\e$ for
large enough $n, m$.  In particular, if $|\mu_\xi|$ is small enough, one can take $C<0$.

Finally, we prove $T_*=\infty$, which shows the existence of $v(t)$ for all $t\ge 0$.
The above estimate on $\Phi(f(t))$ implies $\|f(t)\|_{H_\th^1}\le M_1(e^{Ct/2}+1)$ 
for $t<T_*$, if $f(0)\in H_\th^1$, for some $M_1>0$.
Indeed, this follows from $\Phi(f)= \frac12 \|\nabla f\|_{L_\th^2}^2$ and Lemma
\ref{lem:equiv-cons} below (equivalence of norms). 
However, by Sobolev's imbedding theorem and 
noting that $H^1\equiv H^1(\T) \simeq H_\th^1$ from $\th\in C(\T)$, we have
$H_\th^1\subset C^{\a-1}, \a-1<\frac12$ and this shows $\|v(t)\|_{C^{\a-1}}\le 
M_2(e^{Ct/2}+1), t<T_*$. 
Therefore, noting that $\|u(t)\|_{C^\a}\le 2\|v(t)\|_{C^{\a-1}}$, we see $\|u(t)\|_{C^\a}\le 2M_2(e^{Ct/2}+1),t<T_*$.
This proves $T_*=\infty$ by Lemma \ref{lem:6.2} at least if $f(0)= \fa(v(0)) 
\th^{-1}\in H_\th^1\simeq H^1$, that is, if $v(0)\in \mathcal{D}$.
\end{proof}

If $|\mu_\xi|$ is small enough, based on the estimate \eqref{eq:2.20} obtained in
Proposition \ref{prop:2.6-B} with $c_*:= -C >0$,
we can show the exponential decay of $v(t)$ to 
the unique stationary solution $\bar v_m$ for each conserved quantity $m$.

\begin{cor}  \label{prop:Global}
Assume $v(0)\in \mathcal{D}$ as in Proposition \ref{prop:2.6-B}.
Then, if $|\mu_\xi|$ is small enough,
$f(t):= \fa(v(t))\th^{-1}$ converges to 
the constant $z_m$ in $H^1(\T)$ exponentially fast as $t\to\infty$:
\begin{align} \label{eq:2.25-1}
\| f(t)-z_m\|_{H^1}\le C e^{-c_*t/2}\| f(0)-z_m\|_{H^1},
\end{align}
where $m=\int_\T v(0,x)dx$.  We also have 
\begin{align} \label{eq:2.26-1}
\| v(t)-\bar v_m\|_{C^{\a-1}}\le C e^{-c_*t/2}\| f(0)-z_m\|_{H^1}.
\end{align}
\end{cor}

\begin{proof}
By Lemma \ref{lem:equiv-cons} below, under the conservation law, we obtain
\begin{align*}
\|f(t)-z_m\|_{H_\th^1}^2 & \le C \|\nabla f(t)\|_{L_\th^2}^2 = 2 C\Phi(f(t)) \\
& \le 2 Ce^{-c_*t} \Phi(f(0)) \le C e^{-c_*t}\|f(0)-z_m\|_{H_\th^1}^2.
\end{align*}
This shows the desired estimate on $\| f(t)-z_m\|_{H^1}$,
since the norm of $H_\th^1$ is equivalent to
that of $H^1$ due to the boundedness of $\eta(x)$. 

In order to give the estimate \eqref{eq:2.26-1} on $v(t)$, we first show  the uniform boundedness of  $\|v(t)\|_{C^{\a -1}}$ in $t\geq 0$, i.e.,
\begin{align} \label{eq:2.27}
\sup_{t\ge 0} \|v(t)\|_{C^{\a -1}}<\infty.
\end{align}
By the assumption \eqref{c-fa} on $\fa$, we have that 
$\|\fa^{-1}(v)\|_{C^{\a -1}} \le C(1+\|v\|_{C^{\a -1}}), v\in C^{\a-1}$. 
Therefore, recalling that
$v(t)=\fa^{-1}(f(t)\th)$ and using  Sobolev's imbedding theorem, we
have that 
\begin{align*}
\|v(t)\|_{C^{\a -1}} & \le C( 1 +\|f(t) \theta\|_{C^{\a -1}}) \\
 & \le C( 1 +\|f(t)\|_{C^{\a -1}} \|\theta\|_{C^{\a -1}} ) 
 \le C( 1 + \|\theta\|_{C^{\a -1}} \|f(t)\|_{H^1}).   
\end{align*}
Recall that the constant $C$ changes from line to line.
So,  by \eqref{eq:2.25-1}, we obtain \eqref{eq:2.27}. 
Thus, by noting \eqref{eq:2.27}, we take a function 
$\widetilde{\fa^{-1}} \in C^2(\R)$ with compact support such that 
$\widetilde{\fa^{-1} } (f(t) \th) = v(t)$ for all $t>0$ and
$\widetilde{\fa^{-1} }(z_m \th) = \bar{v}_m(=\fa^{-1}(z_m \th))$. Then, Lemma 9 of \cite{BDH-19} gives that 
\begin{align*}
\| v(t)-\bar v_m\|_{C^{\a-1}}
& =\| \widetilde{\fa^{-1}}  (f(t) \th)  -\widetilde{\fa^{-1} }(z_m \th)\|_{C^{\a-1}}  \\
&\le \|\widetilde{\fa^{-1}} \|_{C^2} (1+\|z_m \th\|_{C^{\a -1}}) 
\| (f(t)-z_m) \th \|_{C^{\a-1}} \\
&\le \|\widetilde{\fa^{-1}} \|_{C^2} (1+\|z_m \th\|_{C^{\a-1}}) \|\th \|_{C^{\a-1}}
\| (f(t)-z_m) \|_{H^1}.
\end{align*}
Consequently, noting that $\|\widetilde{\fa^{-1}} \|_{C^2}$ is independent of $t\geq 0$, we obtain the desired result \eqref{eq:2.26-1} by \eqref{eq:2.25-1}.
\end{proof}

The following lemma was used in the proof of Proposition \ref{prop:2.6-B}
and Corollary \ref{prop:Global}.

\begin{lem}  \label{lem:equiv-cons}
Assume the following condition for $f=f(x)\in H_\th^1$, which comes from the conservation law:
\begin{equation}  \label{eq:f-cons}
\int_\T\big(\fa^{-1}(f \th)-\fa^{-1}(z_m\th)\big) dx=0.
\end{equation}
Then, we have
\begin{equation}  \label{eq:2.21-B}
\|f-z_m\|_{L^2_\th} \le C \|\nabla f\|_{L^2_\th},
\end{equation}
for some $C>0$.
In particular, under the condition \eqref{eq:f-cons}, the Sobolev norm
$\|f-z_m\|_{H^1_\th}$ is equivalent to $\|\nabla f\|_{L^2_\th}$.
\end{lem}

\begin{proof}
First, note that, by mean value theorem applied for \eqref{eq:f-cons},
we see $z_m=f(y_*)$ for some $y_*\in\T$.  Indeed, by noting the
monotone increasing property of $\fa^{-1}$, we see
$$
\int_\T\fa^{-1}(\min f \cdot \th) dx\le \int_\T\fa^{-1}(f\th) dx
\, \Big(\!\!\!= \int_\T\fa^{-1}(z_m \th) dx\Big)
\le \int_\T\fa^{-1}(\max f \cdot \th) dx
$$
and this implies $\min f \le z_m \le \max f$.
Therefore, we have
\begin{align*}
\| f-z_m\|_{L^2} & = \Big(\int_\T \big(f(x)-f(y_*)\big)^2dx \Big)^{1/2} \\
& = \Big(\int_\T dx\Big\{\int_{y_*}^x\nabla f(y')dy' \Big\}^2\Big)^{1/2} \\
&\le \|\nabla f\|_{L^2}.
\end{align*}
Since $0< \th \in C(\T)$, this implies \eqref{eq:2.21-B}.

The equivalence of $\|\nabla f\|_{L^2_\th}$ to the Sobolev norm is now clear as
$$
\|\nabla f\|_{L^2_\th}\le 
 \|f-z_m\|_{H^1_\th} = \big( \|\nabla(f-z_m)\|_{L^2_\th}^2
 + \|f-z_m\|_{L^2_\th}^2\big)^{1/2} \le C \|\nabla f\|_{L^2_\th}.
$$
\end{proof}

The next lemma gives an initial layer type result. 
This is used to remove the condition $v(0)\in \mathcal{D}$ in
Proposition \ref{prop:2.6-B} and Corollary \ref{prop:Global}.  Recall that we only assume 
$\xi\in C^{\a-2}(\T), \a\in (\frac{13}9,\frac32)$.

\begin{lem}\label{lem:Initial layer}
For every initial value $v(0)\in C^{\a-1}$ and all $t\in (0,T_*)$,
the solution $v(t)\in C^{\a-1}$ 
of the SPDE \eqref{eq:0-v} in paracontrolled
sense 
satisfies $f(t) := \fa(v(t))\th^{-1} \in H^1$, that is,
$v(t)\in \mathcal{D}$.  In other words, even if $f(0)\notin H^1$,
immediately after, we have $f(t)\in H^1, t>0$ and this proves $T_*=\infty$
by Proposition \ref{prop:2.6-B}.
\end{lem}
\begin{proof}
If $\xi$ is smooth, taking any $u_0$ such that $\nabla u_0+m=v_0$ for given $v_0$,
we have from \eqref{eq:1.1} with $a=\fa'(\cdot+m)$ and $g=\fa(\cdot+m)$ 
$$
\partial_t u= \fa'(v)\nabla v+\fa(v)\xi = \th \nabla f + \mu f.
$$
This is an integrated form of \eqref{eq:3.1}.  Therefore, we have
$$
u(t)-u(0)= \th\int_0^t \nabla f(s) ds + \mu \int_0^t f(s)ds
$$
or we can rewrite this as
\begin{align}  \label{eq:int-nabla_f}
\int_0^t \nabla f(s,x) ds = \th(x)^{-1} \Big\{u(t,x)-u(0,x)- \mu \int_0^t f(s,
x )ds \Big\}.
\end{align}
The right hand side belongs to $C^{\a-1}$, since $\th^{-1} \in C^{\a-1}$, $u(t)\in C^\a$
and $f(s)\in C([0,T],C^{\a-1})$.
On the other hand, in the left hand side, $\nabla f(s) \in C([0,T], C^{\a-2})$, $T<T_*$,
since $\fa(v(s))\in \mathcal{L}_T^{\a-1}, \th^{-1}\in C^{\a-1}$. 
Therefore, taking the
limit in $\hat\xi$, we have \eqref{eq:int-nabla_f}
for $t<T_*$, if we interpret the left hand side as a Bochner integral in $C^{\a-2}$.

First note that
$$
\int_0^t \nabla f(s,x) ds= \nabla\int_0^t  f(s,x) ds,
$$
where the integrals are Bochner integrals in $C^{\a-2}$ for the left hand side
and $C^{\a-1}$ for the right hand side.  (This can be shown by regarding both sides
as generalized functions and by multiplying test function $\psi$.)
Thus, \eqref{eq:int-nabla_f} implies $\nabla\int_0^t  f(s,x) ds \in 
C([0,T],C^{\a-1})$, since the right hand side of \eqref{eq:int-nabla_f} is in this class.
This shows, by also noting an obvious relation $\int_0^t  f(s,x) ds \in C([0,T],C^{\a-1})$,
that
$$
\int_0^t  f(s,x) ds \in C([0,T],C^{\a}).
$$
Recall that the left hand side is defined as a Bochner integral in $C^{\a-1}$
and $\hat \xi$ is already general.

Since $f\in C([0,T], C^{\a-1})$, for every $t\in (0,T_*)$, by mean value theorem
applied in the space $C^{\a-1}$, we see that there exists $\t=\t_t \in (0,t)$ such that
\begin{align} \label{eq:2.30}
\frac1t \int_0^t  f(s,x) ds= f(\t_t)
\end{align}
holds.  However, since the left hand side belongs to
$C^\a$ and $\a \in (\frac{4}3, \frac32)$, we see $f(\t_t)\in C^\a$ and, in particular,
$f(\t_t)\in H^1$.  Once we have $f(\t_t)\in H^1$, taking the initial value
of $v$ as $v(\t_t)$, by Proposition \ref{prop:2.6-B}, we see $f(s)\in H^1$
for every $s\ge \t_t$ and the solution exists globally in time.
Since we can take $t\in (0,T_*)$ arbitrary small, this shows the conclusion,
recalling the flow property, see Remark \ref{rem-2.4-0}. 
\end{proof} 

\begin{rem}
We can decompose $v$ through $\fa$ as $\fa(v)=f\cdot \th_\xi$
and the bad regularity $v\in C^{\a-1}$ comes only from $\eta$ in $\th_\xi$ and,
by Lemma \ref{lem:Initial layer}, $f$ is a good part.  Recall that 
Tsatsoulis and Weber \cite{TW} 
decomposed the solution of $P(\phi)$-dynamics on $\T^2$ into 
the sum of OU-part and good part.  Since we deal with quasilinear
equation, our decomposition is
nonlinear, but in a sense, similar to this.
\end{rem}

Now we can  easily complete the proof of Theorem \ref{thm:1.1}.
\begin{proof}[Proof of Theorem \ref{thm:1.1}]
The conclusion follows immediately by combining Proposition 
\ref{prop:2.6-B}, Corollary \ref{prop:Global}
and Lemma \ref{lem:Initial layer}. 
Here we only give a short explanation for \eqref{eq:exp-1.9}. 
We first note that, since $v\in C([0,\infty), C^{\a -1})$, we have 
\begin{align}\label{eq:2.32}
\sup_{t\in [0,1]}\|v(t) -\bar{v}_m\|_{C^{\a-1}}< \infty.
\end{align}  On the other hand, by Lemma \ref{lem:Initial layer}, we have $f(1)=\fa(v(1))
\theta^{-1} \in H^1$.  In particular, $v(1)\equiv v(1, v_0) \in \mathcal{D}$ so that
the condition of Corollary \ref{prop:Global} is satisfied.  Therefore, 
by the flow property and \eqref{eq:2.26-1}, for $t>1$, we obtain
\begin{align*}
\|v(t) -\bar{v}_m\|_{C^{\a-1}} & = \|v(t-1, v(1, v_0)) -\bar{v}_m\|_{C^{\a-1}}  \\
& \le C e^{-c_* (t-1)/2} \|f(1) -\bar{v}_m\|_{H^1}, \ t>1,
\end{align*} 
which implies \eqref{eq:exp-1.9} by noting  \eqref{eq:2.32}.
\end{proof}

\subsection{Global existence and asymptotic behavior of $u(t)$ as $t\to\infty$}
\label{sec:2.3-u}

We can apply the result for $v(t)$ to study the global-in-time existence and
the asymptotic behavior of $u(t)$.
First we note that $u(t)$ can be recovered from $v(t):=\nabla u(t)$
at least when $\xi\in C^\infty(\T)$, cf.\ Subsection \ref{sec:6.1}.
Note that $m=0$ under this choice.

\begin{lem} \label{lem:2.Section2.3}
Assume $\xi\in C^\infty(\T)$ and
let the initial value $u_0\in C^{\a}, \a\in (\frac43,\frac32)$
of  \eqref{eq:1.1} be given.  We determine $v(t)$ by solving
\eqref{eq:2} with initial value $v_0:= \nabla u_0$, and set
\begin{equation}  \label{eq:vtou}
u(t,x):= \int_0^x v(t,y)dy+\int_\T u_0(y)dy -\int_\T (1-y)v(t,y)dy
+ \int_0^tds \int_\T \chi(v(s,y))\xi(y)dy,
\end{equation}
where we regard $\T=[0,1)$ especially in the first and third terms of the 
right hand side.
Then, $u(t)$ solves the  SPDE \eqref{eq:1.1} with $a=\fa'$ and $g=\chi$.
\end{lem}

\begin{proof}
First, note that $u(t,x)$ defined by \eqref{eq:vtou} satisfies the initial condition:
$$
u(0,x)= \int_0^x \nabla u_0(y)dy + \int_\T u_0(y)dy-\int_\T (1-y)\nabla u_0(y)dy
= u_0(x).
$$
To see that it satisfies the SPDE \eqref{eq:1.1}, writing the sum of the second to fourth terms 
in the right hand of \eqref{eq:vtou} as $A(t)$, we have
\begin{align*}
\partial_t u(t,x) &= \int_0^x \partial_t v(t,y)dy+ \partial_t A(t) \\
& = \int_0^x [\De\{\fa(v)\}+ \nabla\{\chi(v)\xi\}] dy + \partial_t A(t)\\
& = \nabla\{\fa(v)\}(t,x) -  \nabla\{\fa(v)\}(t,0)+ \chi(v(t,x))\xi(x) -\chi(v(t,0))\xi(0)
+  \partial_t A(t).
\end{align*}
Since $\nabla u=v$ holds for $u$ defined by \eqref{eq:vtou}, we have
$\nabla\{\fa(v)\}= a(\nabla u)\De u$ and $\chi(v)\xi= g(\nabla u)\xi$.
Therefore, to complete the proof of the lemma, it is enough to show
\begin{equation}  \label{eq:pta-2}
\partial_t A(t) = \nabla\{\fa(v)\}(t,0)+ \chi(v(t,0))\xi(0).
\end{equation}
However,
\begin{align*}
\partial_t A(t) &
= -\int_\T (1-y)\partial_tv(t,y)dy + \int_\T \chi(v(t,y))\xi(y)dy \\
& = -\int_\T (1-y) [\De\{\fa(v)\}+ \nabla\{\chi(v)\xi\}] dy + \int_\T \chi(v(t,y))\xi(y)dy \\
& = -\int_\T [\nabla\{\fa(v)\}+ \chi(v)\xi] dy 
-  (1-y) [\nabla\{\fa(v)\}+ \chi(v)\xi]\Big|_0^1
+ \int_\T \chi(v(t,y))\xi(y)dy \\
&= [\nabla\{\fa(v)\}+ \chi(v)\xi](t,0),
\end{align*}
which shows \eqref{eq:pta-2}.
\end{proof}

\begin{rem} \label{rem-2.6-21522}
When $\xi\in C^{\a-2}$, the last term in the right hand side of \eqref{eq:vtou} 
is well-defined even for general $\chi$ in the sense that the other 
four terms are all well-defined.  But, to look at it by itself especially
without time integral, the product $\chi(v(s)) \xi$ is ill-posed
in a classical sense, since $\chi(v(s))\in C^{\a-1}$ and $\xi\in C^{\a-2}$
with $\a\in (\frac43,\frac32)$.  However, when $\chi=\fa$, this product turns out
to be well-defined.  Indeed, in this case, $\fa(v(t))=f(t)\th$ with $f(t)\in H^1$ 
for every $t>0$ by Lemma \ref{lem:Initial layer}.
Accordingly, the last term of \eqref{eq:vtou} without time integral
can be interpreted as ${}_{H^1}\lan f(t), \th \xi\ran_{H^{-1}}$,
$t>0$ if $\th\xi$ can be regarded as an element of $H^{-1}$. 
Noting that $\th \in C^{\a-1}$, $\xi\in C^{\a-2}$ and $(\a-1)+ (\a-2)<0$, in general, 
the product $\th\xi$ is ill-posed. However, in our special case,
we will see that  $\th\xi$ is well-defined as an element of $C^{\a -2}$ and then, in particular,  $\th\xi \in H^{-1}$ because  of $C^{\a-2} \subset H^{-1}$. This gives the motivation for the following proof of Theorem \ref{thm:1.1-u(t)}.

\end{rem}

We now give the proof of Theorem \ref{thm:1.1-u(t)}.  

\begin{proof}[Proof of Theorem \ref{thm:1.1-u(t)}]
The global-in-time existence of $u(t)$ is clear from $T_*=\infty$ noted in
Lemma \ref{lem:Initial layer}.  In the rest, we assume $|\mu_\xi|$ is sufficiently 
small as in Theorem \ref{thm:1.1}.

As we mentioned in Remark \ref{rem-2.6-21522},  we have to show that  the term ${}_{H^1}\lan f(t), \th \xi\ran_{H^{-1}}$ (or equivalently the term $\th \xi$) is well-defined. 
For this, we apply Lemma \ref{lem:2.Section2.3} by introducing an approximation of $u(t)$ as in the proof of Proposition \ref{prop:2.6-B}. Since Lemma \ref{lem:2.Section2.3} holds for $v_0 \in C^{\a-1}$,  in the following, we consider the initial values $u_n(0) =u_0\in C^{\a} $ for all $n$.  Note that, we just have $v(0)=\nabla u_0 \in C^{\a -1}$, which does not imply $v(0) \in \mathcal{D}$. 

Take a sequence of enhanced noises
$\hat \xi_n$ such that $\xi_n\in C^\infty(\T)$ and $\hat \xi_n$ converges to $\hat \xi$
in $C^{\a-2}\times C^{2\a-3}$ as $n\to\infty$.  Then, the associated solution
$u_n$ converges to $u$ in $\mathcal{L}_T^\a$.  Since $T_*=\infty$, one can take
$T>0$ arbitrarily large and, in particular, we have for every $t\ge 0$,
\begin{align}  \label{eq:unu}
\|u_n(t)-u(t)\|_{C^\a} \to 0 \quad (n\to \infty).
\end{align}

Since $\xi_n\in C^\infty(\T)$, by Lemma \ref{lem:2.Section2.3},
we have the formula \eqref{eq:vtou} for $u_n(t,x)$ by replacing $v,\xi$
by $v_n:=\nabla u_n, \xi_n$, respectively, in the right hand side, which
we denote as $A_1^n(t,x)+A_2-A_3^n(t)+A_4^n(t)$.  Recall that we take $a=\fa'$ and $g=\chi=\fa$ here.
We also denote $f_n(t,x):=\fa(v_n(t,x))\th_n^{-1}(x)$ and $\th_n :=\th_{\xi_n}$.
Then, we see $f_n(t,x) \to f(t,x):=\fa(v(t,x))\th^{-1}(x)$ in $\mathcal{L}_T^{\a-1}$ as before, where  $ \th=\th_{\xi}$.

Noting that $\xi_n \in C^\infty(\T)$ and \eqref{th-mu}, we have $\th_n \xi_n=\mu_n -\nabla\th_n$, which are well-defined as the elements of $C^{\a-2}$ and then
\begin{align*}
\|\th_n \xi_n -\th_m \xi_m\|_{C^{\a -2}} = &  \|\mu_{n}-\nabla\th_n -(\mu_{m}-\nabla\th_m)\|_{C^{\a-2}} \\
\le &  |\mu_n- \mu_m| + \|\nabla \th_n -\nabla \th_m \|_{C^{\a -2}} \\
\le & |\mu_n- \mu_m| +  C \|\th_n -\th_m \|_{C^{\a-1}},
\end{align*}
where $\mu_n:=\mu_{\xi_n}$.
Therefore, we see that $\th_n\xi_n$ is a Cauchy sequence of $C^{\a-2}$ by noting that $\th_n \to \th$ in $C^{\a-1}$ and $\mu_n \to \mu=\mu_\xi$ as $n\to\infty$. In the sequel, we denote by  $\th \xi$ the limit of  the sequence $\th_n\xi_n$ in $C^{\a -2}$. So, we have  $\th \xi \in C^{\a-2}$ and in particular, $\th \xi \in H^{-1}$. 

Using the above notation, we claim that
\begin{align} \label{eq-2.36-21522}
& \lim_{n\to \infty } \{A_1^n(t,x) - A_3^n(t) + A_4^n(t)\} \\
=&\int_0^x v(t,y)dy -\int_\T (1-y)v(t,y)dy
+ \int_0^t {}_{H^1} \langle f(s), \th\xi \rangle_{H^{-1}} ds
\notag
\\
=: & A_1(t,x) -A_3(t)+A_4(t).  \notag
\end{align} 
By Lemma \ref{lem:Initial layer}, we know $f(t)\in H^1$ and $\int_0^t f(s,x)ds \in H^1$ for $t>0$. So, ${}_{H^1}\langle f(t), \th \xi \rangle_{H^{-1}}, t>0$ is   well-defined.  In particular, we see  $A_4(t)$ is well-defined under the assumption $a=g'$. 
Let us first show 
\begin{align} \label{eq-2.27-21522}
 \lim_{n\to \infty}A_4^n(t) = A_4(t).
\end{align} 
In order to do it, we show $\int_0^t f_n(s,x)ds$ converges to $\int_0^t f(s,x)ds$ in $H^1$ for $t>0$. 
By \eqref{eq:int-nabla_f} and $\th_n \to \th$ in $C^{\a-1}$ (in particular in $L^\infty$), 
$\inf_{x,n}\th_n(x)>0, \inf_x \th(x)>0$, we have
\begin{align*}
 & \left\|\na \int_0^t f_n(s)ds - \na \int_0^t f(s)ds \right\|_{L^{2}}  \\ 
\le & 
\|\th_n^{-1} (u_n(t) -u(t)) \|_{L^2 } +  \left\|(\th_n^{-1} - \th^{-1}) 
 \Big(u(t) - u_0 + \mu \int_0^t f(s)ds \Big) \right\|_{L^2}  \\
 & \quad + \left\|\th_n^{-1} \mu_n \int_0^t f_n(s)ds -
\th_n^{-1}  \mu \int_0^t f(s)ds \right\|_{L^2} \\ 
\le & 
C\left\{  \|u_n(t) -u(t)\|_{L^2 } +  \left\|(\th_n - \th) \Big(u(t) - u_0 +\mu \int_0^t f(s)ds \Big) \right\|_{L^2} \right.  \\
 &\left.  \quad + \left\|\mu_n \int_0^t f_n(s)ds -
\mu \int_0^t f(s)ds \right\|_{L^2} 
\right\}.
\end{align*}
Since  $f_n(t) \to f(t) \in \mathcal{L}_T^{\a -1}, T>0$, we have 
$\lim_{n\to \infty}\|\int_0^t f_n(s)ds - \int_0^t f(s)ds\|_{L^2}=0$.  Taking $n\to \infty$ in both sides of the above inequality, we have  
$\lim_{n\to \infty} \|\na\int_0^t f_n(s)ds - \na \int_0^t f(s)ds\|_{L^2}^2=0$.
Therefore, we  obtain the desired result, from which \eqref{eq-2.27-21522} follows by noting that
$A_4^n(t)=  \int_0^tds \int_\T f_n(s,y)\th_n(y)\xi_n(y)dy =\int_0^t {}_{H^1}\langle f_n(s), \th_n \xi_n \rangle_{H^{-1}}ds$.

For $A_1^n(t,x)$, we have $A_1^n(t,x) \to A_1(t,x)$ in $C([0,T], C^{\a}), T>0$ as $n \to \infty$. Indeed, noting that $C^\a$ coincides with  the usual $\a$-H\"{o}lder space for $\a\in (\frac{13}9, \frac32)$, we have
\begin{align*}
 \|A_1^n(t)-A_1(t)\|_{C^\a}  = &\|A_1^n(t) - A_1(t)\|_{L^\infty} + \| \nabla A_1^n(t) -\nabla A_1(t)\|_{L^\infty} \\
&  
+\sup_{x\neq y\in \T}\frac{|\nabla (A_1^n -A_1) (t,x)-\nabla (A_1^n-A_1) (t, y)|}{|x-y|^{\a-1}} \\
= & \|A_1^n(t) -A_1(t)\|_{L^\infty} + \|v_n(t)- v(t)\|_{C^{\a-1}}\\
 \leq & \|v_n(t)-v(t)\|_{L^\infty} + \|v_n(t)-v(t)\|_{C^{\a-1}}\\
 \le & 2\|v_n(t) -v(t)\|_{C^{\a-1}},
\end{align*}
where $\nabla A_1^n(t)=v_n(t) $ and $\nabla A_1(t) =v(t)$ have been used for the second equation.
Therefore, noting that $v_n(t) \to v(t)$ in $\mathcal{L}_T^{\a-1}$, we obtain the desired result. Finally for $A_3^n(t)$, by similar arguments, 
 it is easy to see $|A_3^n (t) -A_3 (t)|\leq \frac12\|v_n(t)-v(t)\|_{L^\infty} 
 \leq \frac12\|v_n(t)-v(t)\|_{C^{\a-1}}$, which implies that 
$A_3^n (t)$ converges uniformly to $A_3 (t)$ on any compact interval $[0,T]$. As a consequence, the proof of \eqref{eq-2.36-21522} is completed.

Combining \eqref{eq:unu} with \eqref{eq-2.36-21522}, we have 
\begin{align*}
u(t,x) =A_1(t,x) +A_2 -A_3(t)+A_4(t).
\end{align*} 
We now show the uniform boundedness of $\|A_1(t)\|_{C^\a}+|A_3(t)|$ in $t\geq 0$, that is, 
\begin{align}  \label{eq:2.A1+3}
\sup_{t\ge 0}\{\|A_1(t)\|_{C^\a}+| A_3(t)|\}<\infty.
\end{align} 
Indeed, this can be easily shown by Theorem \ref{thm:1.1}.  By similar arguments for convergences above, we easily have 
$$
\|A_1(t) \|_{C^\a}+|A_3(t)| \leq C\|v(t)\|_{C^{\a-1}}, \ t\ge 0.
$$ 
Noting that  $\int_\T v_0(y)dy= \int_\T \nabla u_0(y)dy =0\, (=m)$ and $|\mu|=|\mu_\xi|$ is sufficiently small, Theorem \ref{thm:1.1}
 gives that $$\|v(t)\|_{C^{\a-1}} \leq \|\bar{v}_0\|_{C^{\a-1}}  + Ce^{-ct}, $$
where $\bar{v}_0$ is defined by \eqref{eq:ST} with $m=0$. Therefore, 
we have $\sup_{t\geq 0}\|v(t)\|_{C^{\a}}<\infty$ and \eqref{eq:2.A1+3} is shown.

Noting that $A_2$ is a constant, in order to show \eqref{eq:unif-u(t)}, it is sufficient to show
\begin{align} \label{eq-2.39-21522}
\sup_{t\geq 0}|A_4(t)-z_0\mu t|<\infty.
\end{align}
Let us rewrite $A_4(t)$ as follows.
\begin{align} \label{eq-2.39-21523}
A_4(t) = z_0 t {}_{H^1}\langle 1, \th\xi \rangle_{H^{-1}} + \int_0^t r(s)ds,
\end{align}
where $z_0$ is determined by \eqref{eq:m-z} with $m=0$ and
$
r(s) :=  {}_{H^1}\langle f(s) -z_0, \th\xi \rangle_{H^{-1}}.
$
By analogous arguments for \eqref{eq:exp-1.9}, see the proof of Theorem \ref{thm:1.1}, we have
\begin{align} \label{eq-2.40-522}
| r(s) |\leq  Ce^{-c_*s/2}\|\th\xi \|_{H^{-1}}\|f(1)-z_0\|_{H^1},\ s \geq 1.
\end{align}
holds for some $c_*>0$, 
where  Corollary \ref{prop:Global} and $f(1) \in H^{1}$ have been used; recall that $|\mu_{\xi}|$ is sufficiently small. 
Noting that $$\left|\int_0^1 r(s)ds \right| \leq 
\left\|\int_0^1( f(s)- z_0)ds \right\|_{H^1} \|\th\xi \|_{H^{-1}}$$ and using \eqref{eq-2.40-522}, we see 
\begin{align*}
\sup_{t\ge 0}\left|\int_0^t r(s)ds \right|<\infty. 
\end{align*}
Finally for the first term in the right hand side of \eqref{eq-2.39-21523},
recalling that $\th\xi$ denotes the limit of $\th_n \xi_n$ in $C^{\a-2}$
and  $\th_n\xi_n=\mu_n-\nabla\th_n$, we see
$${}_{H^1}\langle 1, \th\xi \rangle_{H^{-1}} =\lim_{n \to \infty}\int_\T \th_n\xi_n dy =\lim_{n \to \infty}\int_\T (\mu_n-\nabla\th_n)dy =\lim_{n\to \infty}\mu_n =\mu,$$
where the periodicity of $\th_n$ has been used for the third equation.
Consequently, summarizing the above estimates, we have  \eqref{eq-2.39-21522} and 
complete the proof of Theorem \ref{thm:1.1-u(t)}.
\end{proof}

\section{Continuity of the solution $u(t)$ of \eqref{eq:1.1} in  initial values} \label{sec:3}

In this section, we give the proof of Theorem \ref{thm:CIN}. As  in \cite{FHSX} and stated in Remark \ref{rem-3.2-R2111}, Theorem  
\ref{thm:CIN} follows from  Theorem \ref{thm:ContInintial} below, which is the main result of this section.
Since this part is  a continuation of \cite{FHSX}, the same notations as in \cite{FHSX} will be mostly used. 
We assume $\b\in (\frac13,\a-1)$, $\ga\in (2\b+1,\a+\b)$, and  use notations $K_0$, $K(\|{\bf u}\|_{\a, \b, \ga})$ and 
$K(\|{\bf u}_1\|_{\a, \b, \ga}, \|{\bf u}_2\|_{\a, \b, \ga})$ introduced at the 
end of Section 2 of \cite{FHSX}, which may change from line to line. We still 
write $a\lesssim b$  for two non-negative functions  $a$ and $b$  if there 
exists a constant $C>0$ independent of the variables under consideration 
such that $a \leq C b$. In addition, to emphasize the initial value, we 
sometimes write  
$K_0(\|u_0\|_{C^\a})$, $K_0(\|u_0^1\|_{C^\a}, \|u_0^2\|_{C^\a})$ for $K_0$.

Noting that Theorem \ref{thm:CIN} is about the continuity of the local-in-time solutions on  initial values and recalling Lemma \ref{lem:1.4},  for simplicity but without loss of generality, we may assume the coefficients $a, g \in C_b^3(\R)$ with \eqref{eq:acC} in the sequel, which are same as in   \cite{FHSX}.

Let  $\widetilde{\bf C}_{\a, \b, \ga}(X)$ be the family of functions controlled by $X$, that is, $$\widetilde{\bf C}_{\a, \b, \ga}(X) := \{ (u',u^\sharp); u= \bar\Pi_{u'} X+ u^\sharp,
\ 
\|(u',u^\sharp)\|_{\a,\b,\ga}<\infty\},$$
where  $ \bar\Pi_{u'} X$ denotes the modified paraproduct of $u'$ and 
$X=(-\De)^{-1} Q \xi$, see Section 2 of \cite{FHSX} and
\begin{equation*}
\|(u',u^\sharp)\|_{\a,\b,\ga} := \|u'\|_{\mathcal{L}_T^\b} + \|u^\sharp\|_{\mathcal{L}_T^\a}
+ \sup_{0<t\le T} t^{\frac{\ga-\a}2} \|u^\sharp(t)\|_{C^{\ga}}.
\end{equation*}
Noting that $u$ is a function of $u'$ and $u^\sharp$, we used the space ${\bf C}_{\a, \b, \ga}(X)$ and  the norm 
$\|(u, u')\|_{\a, \b, \ga}$ in \cite{FHSX}, which are equivalent to the above.
For the given initial value $u_0$, we impose additional conditions
$u(0)=u_0,\ u'(0)= \frac{g(\nabla u_0)}{a(\nabla u_0)}$
and define 
the (skew product) space $\mathcal{X}$  by
$$
\mathcal{X} := \Big\{(u_0,u',u^\sharp) \in C^\a \times \widetilde{\bf C}_{\a, \b, \ga};
(u',u^\sharp) \text{ satisfies } u(0)=u_0, u'(0)= \frac{g(\nabla u_0)}{a(\nabla u_0)} \Big\},
$$
 which is viewed as a metric space embedded in the Banach space
$C^\a \times \widetilde{\bf C}_{\a, \b, \ga}$ equipped with the norm
$$
\|(u_0,{\bf u})\|_{C^\a \times \widetilde{\bf C}_{\a, \b, \ga}}
:= \|u_0\|_{C^\a} + \|{\bf u}\|_{\widetilde{\bf C}_{\a, \b, \ga}},
\ \  {\bf u}=(u',u^\sharp).
$$
Let $\mathcal{L}^0 := \partial_t -a(\nabla u_0^T)\De$,  
$u_0^T:= e^{T\De}u_0$, where $e^{t\De}$ denotes the semigroup generated by $\De$ 
on $\T$. Let us  now define the  map $\Phi$ from $\mathcal{X}$ into itself by 
\begin{align}\label{Phi-6.5}
\Phi:(u_0,u',u^\sharp) \in \mathcal{X}
\mapsto (u_0, v',v^\sharp) \in \mathcal{X},
\end{align}  
where 
\begin{align} 
 v'= & \frac{g(\nabla u) - \big(a(\nabla u)-a(\nabla u_0^T)\big) u'}{a(\nabla u_0^T)}, 
\label{eq:2.17-F}  \\
\label{eq:2.18-F} \mathcal{L}^0v^\sharp= & \Pi_{a(\nabla u_0^T)v'} \xi -\mathcal{L}^0(\bar{\Pi}_{z'} X)
+g'(\nabla u)\Pi(\nabla u^\sharp,\xi)-a'(\nabla u)\Pi(\nabla u^\sharp,\bar\Pi_{u'}\xi)\\
\notag&\qquad+\big(a(\nabla u)-a(\nabla u_0^T)\big)\De u^\sharp
+\zeta \\
=:&   \Pi_{a(\nabla u_0^T)v'} \xi -\mathcal{L}_0(\bar{\Pi}_z X) + \e_1(u,u')+\e_2(u,u')+\zeta(u,u'),
\notag
\end{align}
where $\e_1$ denotes the difference of the third and the fourth terms, $\e_2$ denotes the fifth term in the right hand of \eqref{eq:2.18-F},  and $\zeta$ denotes the remainder, see Lemma 3.3 and  (2.32)  of  \cite{FHSX} for details.
\begin{rem}
The map $\Phi$ defined by \eqref{Phi-6.5} is a little different from the original one introduced in \cite{FHSX}, where the map $\Phi : (u,u')\to (v, v') $  defined on ${\bf C}_{\a, \b, \ga}(X) := \{ (u,u'); u= \bar\Pi_{u'} X+ u^\sharp,
\|(u,u')\|_{\a,\b,\ga}<\infty\}$ (more precisely on the space $\mathcal{B}_T(\la)$,  see  Subsection 3.1 of \cite{FHSX} or equivalently \eqref{3.4-R21026}  below) is introduced, see (2.16)-(2.18) in  \cite{FHSX}.  As we pointed out at the end of Section 3 of \cite{FHSX}, the space
${\bf C}_{\a, \b, \ga}(X)$ is actually identical to that of pairs $(u',u^\sharp)$
with the norm $\|(u,u')\|_{\a,\b,\ga}$ defined in terms of $(u',u^\sharp)$.
In this way, ${\bf C}_{\a, \b, \ga}(X)$ is identified with
$\widetilde{\bf C}_{\a, \b, \ga}$. Therefore, Theorem 3.1 of \cite{FHSX} holds  for $\Phi$ defined by \eqref{Phi-6.5} whenever the initial value is fixed and we will not repeat this fact in the sequel.
\end{rem}

For   $\lambda >0$ and $T>0$, set  
\begin{align} \label{3.4-R21026}
\mathcal{B}_T(\la) := \Big\{ ({u}_0, {\bf u})\in \mathcal{X};\ 
u_0\in C^\a, \  \|{\bf u}\|_{\a, \b,\ga} \leq \la
\Big\}.
\end{align}

We recall that for each fixed initial value $u_0$, in  \cite{FHSX}, it is proved that the map $\Phi$ is contractive 
on $\mathcal{B}_T(\la)$ for  a large enough $\la$ and a small time  $T>0$, and the unique fixed point on  $\mathcal{B}_T(\la)$ solves the paracontrolled SPDE  \eqref{eq:1.1} up to time $T>0$, see Theorem 3.1 of \cite{FHSX}  or Theorem \ref{thm-2.4-R21120} for details. In particular, one explicit choice of $\lambda$ and $T$ is given in  Theorem \ref{thm-2.4-R21120}.
Let us now give the main result of this section.
\begin{thm} \label{thm:ContInintial}
Let  $\a \in (\frac{13}{9}, \frac{3}{2})$. Then, we have that 
the map  $\Phi$ defined by \eqref{Phi-6.5} depends continuously on the initial value $u_0$ and its contractivity on  $\mathcal{B}_T(\la), \ T>0$ is locally uniform in initial values.  
More precisely,  there exists a unique continuous map 
$u_0\in  C^\a \mapsto 
(u'({u}_0), u^\sharp({u}_0) ) \in  \widetilde{\bf C}_{\a, \b, \ga}$ up to time $T$ such that 
$$\Phi({u}_0, u'({u}_0), u^\sharp({u}_0)) =({u}_0,  u'({u}_0), u^\sharp(u_0))$$ and  $(u'({u}_0), u^\sharp({u}_0))$ (or equivalently $( u({u}_0), u'({u}_0) )$) solves the SPDE  \eqref{eq:1.1} starting from ${u}_0$ up to  time $T$ in the paracontrolled  sense. Moreover, $T$ can be taken depending continuously on $\|u_0\|_{C^\a}$ and  $\|\hat{\xi}\|_{C^{\a-2} \times C^{2\a -3}}$.
\end{thm} 
\begin{rem} \label{rem-3.2-R2111}
{\rm(i)} Combining Theorem 3.1-{\rm(ii)} of \cite{FHSX} with Theorem \ref{thm:ContInintial} above, we see that the map $\Phi$ depends continuously on both the initial value $u_0\in C^\a$ and the enhanced noise $\hat{\xi}$. In particular, the unique fixed point of $\Phi$ in $\mathcal{B}_T(\lambda)$ inherits the continuity in $(\hat{\xi},u_0)$. \\
{\rm (ii)} Theorem \ref{thm:CIN}  immediately follows from Theorem \ref{thm:ContInintial} and the relation between $\Phi$ and the solution $u(t)$, see the proof  of Theorem 1.1 in \cite{FHSX} for detailed explanation.\\ 
{\rm (iii)} To obtain the estimate \eqref{6.12-R2613} in the proof of Lemma \ref{lem-6.6-R2616} below, which is important for the proof of Theorem 
\ref{thm:ContInintial},  the original assumption $\a\in (\frac{4}3, \frac32)$ 
in \cite{FHSX} is changed to $\a \in (\frac{13}{9}, \frac{3}{2})$ due to our techniques. Although  the assumption on $\a$ becomes slightly restrictive, the most important case of the spatial white noise on $\T$ is covered.
\end{rem}
We can  obtain immediately the next result by Theorem \ref{thm:ContInintial}.
\begin{cor} \label{lem:6.1} 
The SPDE \eqref{eq:1.1} is
solvable up to time $T>0$ and one can take
$T=T(\|u_0\|_{C^\a}, \|\hat\xi\|_{C^{\a-2}\times C^{2\a-3}})$, which depends continuously on $\|u_0\|_{C^\a}$ and  $\|\hat\xi\|_{C^{\a-2}\times C^{2\a-3}}$.
\end{cor}

To prove Theorem \ref{thm:ContInintial}, thanks to Theorem 3.1 in \cite{FHSX} and the implicit function theorem, the main task is to show the continuity of the map $\Phi$. 
We will first give some lemmas as preparation and postpone the proof of Theorem \ref{thm:ContInintial} to the end of this section.

Since the main purpose of this section is to prove the continuity of $u(t)$ in its initial values, we have to consider different initial values. So, we first generalize (2.15) in Lemma 2.9 of \cite{FHSX} to the next lemma.
\begin{lem}  \label{lem-6.2-R267}
Let $F\in C_b^2(\R)$, $\a \in (1,2)$ and $\b \in (0, \a-1]$. Then, for any 
$u, v \in \mathcal{L}_T^\a$ (without the restriction $u(0)=v(0)$), we have
\begin{align*}
 \|F(\nabla u) -F(\nabla v)\|_{\mathcal{L}_T^\b }\lesssim &
T^{\frac{\a-\b-1}{2}} \|F\|_{C^2}(1+\|u\|_{\mathcal{L}_T^\a})
\|u-v\|_{\mathcal{L}_T^\a} \\
& +\|F\|_{C^2}(1+ \|u(0)\|_{C^\a}) \| u(0) -v(0) \|_{C^\a}.
\end{align*}
\end{lem}
\begin{proof}
This lemma can be easily shown by modifying the proof of Lemma 9 of \cite{BDH-19}.
Although we do not assume $u(0)=v(0)$, refining the proof  of Lemma 9 of \cite{BDH-19}, we deduce that 
 \begin{align*}
 \|F(\nabla u) -F(\nabla v)\|_{\mathcal{L}_T^\b }\lesssim &
T^{\frac{\a-\b-1}{2}} \|F\|_{C^2}
(1+\|\nabla u\|_{\mathcal{L}_T^{\a-1 }}) \|\nabla u-\nabla v\|_{\mathcal{L}_T^{\a -1}}  
\\
&
+\|F(\nabla u(0)) -F(\nabla v(0))\|_{C^\b},
\end{align*}
which gives the desired result by Lemmas 2.3 and 2.9 of \cite{FHSX}.
\end{proof}
As an application of Lemma \ref{lem-6.2-R267}, we have the following lemma, which will be used frequently.
\begin{lem} \label{lem-6.2-R2616}
For any $(u_0^i, u_i', u_i^\sharp) \in \mathcal{B}_T(\lambda), i=1,2$, we have
\begin{align} \label{eq-6.11-R2616}
& \left\|a(\nabla u_1) -a(\nabla u_{0}^{1,T})  -\Big(a(\nabla u_2) -a(\nabla u_{0}^{2,T}) \Big)  \right\|_{\mathcal{L}_T^\b} \\
\lesssim &  
T^{\frac{\a-\b-1}{2}} K(\|{\bf u}_1\|_{\a, \b, \ga}) (1+\|\xi \|_{C^{\a-2} })^2
 \|{\bf u}_1 -{\bf u}_2\|_{\a, \b, \ga}
 +K_0(\|u_0^1\|_{C^\a})\| u_0^1 -u_0^2 \|_{C^\a}. \notag
\end{align} 
\end{lem}
\begin{proof}
By Lemma 2.9-{\rm(i)} and Lemma 3.2-{\rm (i)} of \cite{FHSX}, we have that  for any $\b \in (0, \a-1]$,
\begin{align} \label{eq-6.11-R267}
\|a(\nabla u_{0}^{1,T}) - a( \nabla u_{0}^{2,T}) \|_{C_\b} 
& \lesssim (1+\|u_{0}^{1,T}\|_{C^{\b +1}})
\|u_{0}^{1,T} -u_{0}^{2,T}\|_{C^{\b+1}}\\
& \lesssim 
(1+\|u_{0}^1\|_{C^{\a}})\|u_0^1 -u_0^2\|_{C^\a}. \notag
\end{align}
On the other hand,  Lemma \ref{lem-6.2-R267} gives that 
\begin{align*}
\left\|a(\nabla u_1)  -a(\nabla u_2)  \right\|_{\mathcal{L}_T^\b} 
\lesssim & T^{\frac{\a-\b-1}{2}} (1+\|u_1\|_{\mathcal{L}_T^\a})
\|u_1- u_2 \|_{\mathcal{L}_T^\a}  + (1+ \|u_0^1\|_{C^\a}) \| u_0^1 -u_0^2\|_{C^\a}. \notag
\end{align*}
Therefore, by the estimate  
$\|u\|_{\mathcal{L}_T^{\a}} \lesssim (1+ \|X\|_{C^\a}) 
\|{\bf u}\|_{\a, \b, \ga}$, see  (3.9) in \cite{FHSX},  together with  $\|X\|_{C^\a}\lesssim 
\|\xi\|_{C^{\a-2}}$, we  immediately obtain \eqref{eq-6.11-R2616} with  $K(\|{\bf u}_1\|_{\a, \b, \ga})= (1+\|{\bf u}_1\|_{\a, \b, \ga})$ and $K_0(\|u_0^1\|_{C^\a}) = (1+ \|u_0^1\|_{C^\a}) $. 
\end{proof}

According to Theorem 3.1 of \cite{FHSX} and its proof,  we have the following result.
\begin{lem} \label{lem-6.3-R2616} 
Let $\Phi$ be the map on $\mathcal{B}_T(\lambda)$ defined by  \eqref{Phi-6.5}, that is,  
$\Phi(u_0, u', u^\sharp)=(u_0, v', v^\sharp)$ for  $(u_0, u', u^\sharp) \in \mathcal{B}_T(\lambda)$.
Then, we  have 
\begin{align*}
\|(v', v^\sharp)\|_{\a, \b,\ga} \lesssim
 T^{\frac{\a + \b -\ga}2} K(\|{\bf u}\|_{\a, \b, \ga})\tilde{K}_1(X, \xi) 
 +K_0(\|u_0\|_{C^\a}) (1+\|\xi\|_{C^{\a -2}}),
\end{align*}
where $v'$ are $v^\sharp$ are determined by   \eqref{eq:2.17-F} and \eqref{eq:2.18-F} respectively, and $\tilde{K}_1(X, \xi)$ denotes the same constant as that introduced in  Proposition 3.7 of \cite{FHSX}.
\end{lem}

By Lemma \ref{lem-6.3-R2616}, in particular, we know that $v'$ is well-defined as an element of $\mathcal{L}_T^\b$ for each $(u_0, u', u^\sharp) \in \mathcal{B}_T(\lambda)$. In the following, we will show the local 
Lipschitz continuity of $v'$ in $(u_0, u', u^\sharp)$.
For   $(u_0^i,u_i',u_i^\sharp)\in \mathcal{B}_T(\lambda), \ i=1,2$,  we set $\Phi(u_0^i,u_i',u_i^\sharp) = (u_0^i, v_i', v_i^\sharp)$ in the following.
\begin{lem} \label{lem-6.4-R2614}
We have that $v'$ is locally Lipschitz in $(u_0, u', u^\sharp) \in \mathcal{B}_T(\lambda)$. More precisely, 
we have
\begin{align*}
\|v_1' -v_2'\|_{\mathcal{L}_T^\b} \lesssim &   T^{\frac{\a-\b-1}{2}} 
K(\|{\bf u}_1\|_{\a, \b, \ga}, \|{\bf u}_2\|_{\a, \b, \ga})
(1+\|\xi \|_{C^{\a-2}})^2  \|{\bf u}_1 -{\bf u}_2\|_{\a, \b, \ga}  \\
&  + K(\|{\bf u}_1\|_{\a, \b, \ga})  K_0(\|u_0^1\|_{C^\a}, \|u_0^2\|_{C^\a}) \|u_0^1 -u_0^2\|_{C^\a},
\end{align*}
where ${\bf u}_i= (u'_i, u^\sharp_i)$ for $i=1,2$.
\end{lem}
\begin{proof}
For simplicity of notation, we set 
\begin{align} \label{eq-3.7-210531}
b(u_i)=a(\nabla u_i) -a(\nabla u_{0}^{i,T}),\  i=1,2.
\end{align} Then, it is easy to know that $\|v_1' -v_2'\|_{\mathcal{L}_T^\b}$ is bounded from above by the sum of the following three terms:
\begin{align*}
I_1=& \left\|\big(g(\nabla u_1) -
b(u_1)u_1' \big)
\Big( \frac{1}{a(\nabla u_{0}^{1, T})}  - \frac{1}{a(\nabla u_{0}^{2, T})} \Big)\right\|_{\mathcal{L}_T^\b}, \\
I_2 = & \left\|\frac{1}{a(\nabla u_{0}^{2, T})} \Big( g(\nabla u_1) 
-g(\nabla u_2) -\big( b(u_1) -b(u_2)\big)  u_1' \Big) \right\|_{\mathcal{L}_T^\b},
\\
I_3 = & \left\|\frac{b(u_2)(u_1' -u_2')}{a(\nabla u_{0}^{2, T})} \right\|_{\mathcal{L}_T^\b}.
\end{align*}
Let us first deal with the first term $I_1$. 
From the proof of Lemma 3.4 of \cite{FHSX}, it easily follows that 
\begin{align*} 
\|g(\nabla u_1) -b(u_1) u_1'\|_{\mathcal{L}_T^\b}  \lesssim
T^{\frac{\a -\b -1}{2}} K(\|{\bf u}_1 \|_{\a, \b, \ga})(1+\|X\|_{C^\a}) + K_0(\|u_0^1\|_{C^\a}).
\end{align*}
Recalling that $a \in C_b^3({\R})$ satisfies \eqref{eq:acC} and using \eqref{eq-6.11-R267} together with Lemma 2.9 of \cite{FHSX}, we have that
\begin{align*}
\left\|\frac{1}{a(\nabla u_{0}^{1, T})}  - \frac{1}{a(\nabla u_{0}^{2, T})}
\right\|_{C^\b}
\lesssim&  
\left\| \frac{1}{a(\nabla u_{0}^{1, T})a(\nabla u_{0}^{2, T})}\right\|_{C^\b} \|a(\nabla u_{0}^{1,T}) - a( \nabla u_{0}^{2,T}) \|_{C_\b} \\
\lesssim & 
(1+\|u_0^1\|_{C^\a} )^2 (1+\|u_0^2\|_{C^\a} ) \|a(\nabla u_{0}^{1,T}) - a( \nabla u_{0}^{2,T}) \|_{C_\b}\\
\lesssim & 
K_0(\|u_0^1\|_{C^\a}, \|u_0^2\|_{C^\a})  \|u_0^1 -u_0^2\|_{C^\a},
\end{align*}
where we have been used  the estimate 
\begin{align}\label{3.8-R290}
\left\|\frac{1}{a(\nabla u_0^{i,T})}\right\|_{C^\b} 
\lesssim  (1+\|u_0^i\|_{C^\a} ), \ i=1,2,
\end{align} see  (3.23) of \cite{FHSX},
for the second inequality.

Therefore, by Lemma 2.1 of \cite{FHSX} and  the above estimates, we obtain 
\begin{align*}
I_1\lesssim &  
\left\|g(\nabla u_1) -
b(u_1) u_1' \right\|_{\mathcal{L}_T^\b} 
\left\|\frac{1}{a(\nabla u_{0}^{1, T})}  - \frac{1}{a(\nabla u_{0}^{2, T})}
\right\|_{C^\b}
\\
\lesssim &
\Big( T^{\frac{\a -\b -1}{2}} K(\|{\bf u}_1\|_{\a, \b, \ga})(1+\|\xi\|_{C^{\a-2} } ) + 1 \Big) K_0(\|u_0^1\|_{C^\a}, \|u_0^2\|_{C^\a})\|u_0^1 -u_0^2\|_{C^\a}.\notag
\end{align*}
Next, let us evaluate the second term $I_2$. By Lemma \ref{lem-6.2-R2616}, we have
\begin{align*}
& \left\|\big(b(u_1) -b(u_2) \big) u_1'  \right\|_{\mathcal{L}_T^\b}
\\
 \lesssim & T^{\frac{\a-\b-1}{2}} K(\|{\bf u}_1\|_{\a, \b, \ga})
 (1+\|\xi\|_{C^{\a -2 } })^2  \|{\bf u}_1 -{\bf u}_2\|_{\a, \b, \ga}
 +K_0(\|u_0^1\|_{C^\a}) \|{\bf u}_1\|_{\a, \b, \ga} \| u_0^1 -u_0^2 \|_{C^\a}. 
\end{align*}
In addition, we have the following estimate more easily.
\begin{align*}
\|g(\nabla u_1) -g(\nabla u_2)\|_{\mathcal{L}_T^\b}
\lesssim &  T^{\frac{\a-\b-1}{2}} K(\|{\bf u}_1\|_{\a, \b, \ga})(1+\| \xi \|_{C^{\a-2} })^2   \|{\bf u}_1 -{\bf u}_2\|_{\a, \b, \ga}  \\
&  +K_0(\|u_0^1\|_{C^\a})\| u_0^1 -u_0^2 \|_{C^\a}.
\end{align*}
Therefore, noting \eqref{3.8-R290}, we have
\begin{align*}
I_2 \lesssim &  T^{\frac{\a-\b-1}{2}} K(\|{\bf u}_1\|_{\a, \b, \ga}, \|u_0^2\|_{C^\a})
 (1+\|\xi \|_{C^{\a-2 } })^2  \|{\bf u}_1 -{\bf u}_2\|_{\a, \b, \ga}
 \\
 & 
+ K_0(\|u_0^1\|_{C^\a}, \|u_0^2\|_{C^\a} )(1+\|{\bf u}_1\|_{\a, \b, \ga})\| u_0^1 -u_0^2 \|_{C^\a}.
\end{align*}
Finally, using (3.17) of \cite{FHSX} together with \eqref{3.8-R290},
we easily have 
\begin{align*}
I_3 \lesssim T^{\frac{\a -\b -1}{2}}  K(\|{\bf u}_2 \|_{\a, \b, \ga})
 (1+\| \xi \|_{C^{\a-2}}) \|{\bf u}_1 -{\bf u}_2\|_{\a, \b, \ga}.
\end{align*}
Consequently, the proof of this lemma is completed by the above estimates.
\end{proof}

In the next two lemmas, we give the bounds for the terms involving both $u_0^T$ or (and) $v^\sharp$, which should be evaluated in the weighted space in time.

\begin{lem} \label{lem-6.3-R2614}
We have
\begin{align} 
 \sup_{0<t \leq T} t^{\frac{\ga -\a}{2}}
\Big\|\Big( \big(a(\nabla u_0^{1,T}) -a(\nabla u_0^{2,T}) \big) \De v_1^\sharp \Big)(t) \Big\|_{C^{\gamma -2}}
 \lesssim K_0(\|u_0^1\|_{C^\a}) \|(v_1', v_1^\sharp)\|_{\a, \b,\ga} \|u_0^1 -u_0^2\|_{C^\a}. \notag
\end{align}

\end{lem}
\begin{proof} 
By Lemma \ref{lem-6.3-R2616}, we know $(v_1', v_1^\sharp) \in \widetilde{\bf C}_{\a, \b, \ga}$. In particular, we have that for $t>0$, $v_1^\sharp\in C^\ga$. 
Then, noting that $0<\b + \ga -2<\b$ and  using  Lemma 2.4 of \cite{FHSX},  we have 
\begin{align*}
\Big\|\Big( \big(a(\nabla u_0^{1,T}) -a(\nabla u_0^{2,T}) \big) \De  v_1^\sharp \Big)(t)
\Big\|_{C^{\ga -2}}
\lesssim &  \Big\|\Big(  \big(a(\nabla u_0^{1,T}) -a(\nabla u_0^{2,T}) \big) \De  v_1^\sharp  \Big)(t)
\Big\|_{C^{\b + \ga -2}}  \\
\lesssim &  \Big\|a(\nabla u_0^{1,T}) -a(\nabla u_0^{2,T})\Big\|_{C^\b} 
\| v_1^\sharp (t) \|_{C^{\ga }}.
\end{align*}
Now, we conclude our proof by \eqref{eq-6.11-R267} together with the fact
$\sup_{0<t \leq T} t^{\frac{\ga -\a}{2}} \| v_1^\sharp (t)  \|_{C^{\ga }} \leq \|(v_1', v_1^\sharp)\|_{\a, \b,\ga}.$
\end{proof}

\begin{lem}
\label{lem-6.6-R2616}
Suppose further $\a \in (\frac{13}{9}, \frac{3}{2})$. Then we have 
\begin{align}\label{6.9-R2613}
&  \sup_{0< t \leq T} t^{\frac{\ga-\a}{2}} \Big\| 
\Big(  
\big( a(\nabla u_{0}^{1,T}) -a(\nabla u_{0}^{2, T} ) \big) \De 
\bar{\Pi }_{v_1'} X +
{\Pi}_{( a(\nabla u_{0}^{1,T}) -a(\nabla u_{0}^{2, T} ) )v_1'} \xi \Big)(t) \Big\|_{C^{\a +\b-2}} \\
 \lesssim & K_0(\|u_0^{1} \|_{C^{\a }}) \|(v_1', v_1^\sharp)\|_{\a,\b,\ga}  
\|\xi\|_{C^{\a-2}} \|u_0^{1}-u_0^{2} \|_{C^{\a}}. \notag
\end{align}
\end{lem}

\begin{proof}
Set $\widetilde{a_0} :=a(\nabla u_{0}^{1, T} ) -a(\nabla u_{0}^{2,T} )$ for the sake of brevity.
Using the the commutator  $R_2(v'_1, X) =[\De, \bar{\Pi}_{v_1'}] X$ and $\De X =-Q\xi = -(\xi -\xi(\T))$, see (2.1) of \cite{FHSX}, we have  
\begin{align}\label{3.10-R2921}
\|R_2(v'_1, X)\|_{C_TC^{\a+ \b-2}} \lesssim \|v_1'\|_{C_TC^\b}\|\xi \|_{C^{\a-2}}
\end{align}
 by Lemma 2.6 of \cite{FHSX}
and
\begin{align*}
\widetilde{a_0} \De \bar{\Pi }_{v_1'} X
= \widetilde{a_0} \big(\bar{\Pi }_{v_1'}(\De X) +R_2(v'_1, X) \big)
=& -\widetilde{a_0} \bar{\Pi }_{v_1'}\xi +\widetilde{a_0} R_2(v'_1, X).
\end{align*}
Then, an analogous argument for (2.26) of \cite{FHSX} shows that
\begin{align*} 
\widetilde{a_0}  \De \bar{\Pi }_{v_1'} X
= -{\Pi }_{\widetilde{a_0}  v_1'}\xi -\Pi(\widetilde{a_0}, \bar{\Pi }_{v_1'} \xi )
-R(\widetilde{a_0}, v_1';\xi) -\Pi_{\bar{\Pi}_{v_1'} \xi} \widetilde{a_0}
+\widetilde{a_0} R_2(v'_1, X),
\end{align*}
see Lemma 2.8 of \cite{FHSX} for the meaning of $R(\widetilde{a_0}, v_1';\xi)$.
Therefore, the term inside of the norm $\|\cdot \|_{C^{\a +\b -2}}$
 of the  left hand side of  \eqref{6.9-R2613} equals to 
 \begin{align}\label{6.10-R2613}
 -\Pi(\widetilde{a_0}, \bar{\Pi }_{v_1'} \xi )
-R(\widetilde{a_0}, v_1';\xi) -\Pi_{\bar{\Pi}_{v_1'} \xi} \widetilde{a_0}
+\widetilde{a_0} R_2(v'_1, X).
 \end{align}
Let us first deal with the first term of \eqref{6.10-R2613}, i.e., the resonant term $\Pi(\widetilde{a_0}, \bar{\Pi }_{v_1'} \xi )$.
By Lemmas 2.9 and  3.2 of \cite{FHSX}, we have for $\a< \ga'\leq \ga$
\begin{align}\label{6.11-R2613}
\|\widetilde{a_0}\|_{C^{\ga'-1}} \lesssim & (1+\| u_0^{1,T} \|_{C^{\ga' }}) \| u_0^{1,T}-  u_0^{2,T} \|_{C^{\ga' }} \\
\lesssim & T^{-\frac{\ga'-\a}{2}}  (1+T^{-\frac{\ga'-\a}{2}}\|u_0^{1} \|_{C^{\a }}) \|u_0^{1}-u_0^{2} \|_{C^{\a}}  \notag\\
\lesssim & T^{-(\ga'-\a)}  (T^{\frac{\ga'-\a}{2}}+\|u_0^{1} \|_{C^{\a }}) \|u_0^{1}-u_0^{2} \|_{C^{\a}}.  \notag
\end{align} 
Let $\ga' =\frac{\ga +\a}{2}$. Then, we have
 $\a+ \ga' -3 >0$ and $\ga' >\a$ by noting that  $\a \in (\frac{13}{9}, \frac{3}{2})$.
Using Lemmas 2.1 and 2.4 of \cite{FHSX} and \eqref{6.11-R2613} with $\ga' =\frac{\ga +\a}{2}$, we have 
\begin{align} \label{6.12-R2613}
 \sup_{0< t \leq T} t^{\frac{\ga-\a}{2}}\| \Pi(\widetilde{a_0}, \bar{\Pi }_{v_1'} \xi )(t)\|_{C^{\a+ \b -2}} 
\lesssim & \sup_{0< t \leq T} t^{\frac{\ga-\a}{2}}\| \Pi(\widetilde{a_0}, \bar{\Pi }_{v_1'} \xi )(t)\|_{C^{\a+\ga' -3}} \\
\lesssim & \sup_{0< t \leq T} t^{\frac{\ga-\a}{2}} \|\widetilde{a_0}\|_{C^{\ga'-1}} \| \bar{\Pi }_{v_1'} \xi (t)\|_{C^{\a -2}} \notag \\
\lesssim &(1+\|u_0^{1} \|_{C^{\a }}) 
\|v_1'\|_{\mathcal{L}_T^\b}\|\xi\|_{C^{\a-2}} \|u_0^{1}-u_0^{2} \|_{C^{\a}} ,  \notag
\end{align}
where we have used $\a+ \ga' -3 >0$ for the first inequality and the relation $\ga' -\a =\frac{\ga- \a}{2}>0$  for the last inequality. 

Next, we deal with the last three terms of 
\eqref{6.10-R2613}. By \eqref{eq-6.11-R267} and  the similar arguments to Lemma 4.3 of \cite{FHSX}, we easily have
\begin{align} \label{6.13-R2613}
& \|-R(\widetilde{a_0}, v_1',\xi) -\Pi_{\bar{\Pi}_{u'} \xi} \widetilde{a_0}
+\widetilde{a_0} R_2(v'_1, X)\|_{C_TC^{\a +\b-2}} \\
\lesssim & 
(1+\|u_{0}^1\|_{C^{\a}}) 
\|v_1'\|_{\mathcal{L}_T^\b} \|\xi\|_{C^{\a-2}} \|u_0^1 -u_0^2\|_{C^\a}. \notag
\end{align}
We omit the details and just give the estimate on the last term $\widetilde{a_0} R_2(v'_1, X)$ as an example.  
Noting the relation $0<2\a+ \b -3< \a -1$ and using \eqref{3.10-R2921}, we have
\begin{align*}
 \|\widetilde{a_0}  R_2(v_1', X)\|_{C_TC^{\a+ \b -2}} \lesssim & 
 \|\widetilde{a_0}  R_2(v_1', X)\|_{C_TC^{2\a+ \b -3}} \\
\lesssim & \|\widetilde{a_0}\|_{C^{\a -1}}
\|R_2(v_1', X)\|_{C_TC^{\a+\b -2}}\\
\lesssim & 
(1+\|u_{0}^1\|_{C^{\a}})
\|v_1'\|_{\mathcal{L}_T^\b} \|\xi\|_{C^{\a-2}} \|u_0^1 -u_0^2\|_{C^\a} ,
\end{align*}
where \eqref{eq-6.11-R267} has been used for the last inequality.
As a consequence, we conclude the proof of this lemma by \eqref{6.12-R2613} and \eqref{6.13-R2613} together with $\ga>\a$.
\end{proof}

Next, we  reevaluate the term $\e_2$ involving  $u^\sharp$ according to our purpose.
\begin{lem} \label{lem-6.8-R2618}
We have the local Lipschitz estimate for $\e_2$:
\begin{align}\label{6.16-R2613}
 & \sup_{0< t \leq T}t^{\frac{\ga-\a}{2}} \|\e_2({\bf u}_1)-\e_2({\bf u}_2) \|_{C^{\ga -2}}   \\
\lesssim & T^{\frac{\a -\b -1}{2}}K(\|{\bf u}_1\|_{\a,\b,\ga}, 
\|{\bf u}_2\|_{\a,\b,\ga}) 
(1+\|\xi\|_{C^{\a -2}})^2 \| {\bf u}_1 - {\bf u}_2\|_{\a, \b, \gamma} \notag \\
&  +K_0(\| u_0^1\|_{C^\a}) \|{\bf u}_1\|_{\a,\b,\ga} \|u_0^1 -u_0^2\|_{C^\a}.
 \notag                                                        
\end{align}
\end{lem}
\begin{proof}
This lemma can be shown essentially by  analogous arguments to 
Lemma 3.3 of \cite{FHSX}. However, in that proof, the initial value is fixed and  the explicit relation between the constants and  its initial value is not written down. Hence, for the reader's convenience, we give proof of  \eqref{6.16-R2613} briefly.  In the following, we use the notation $b(u_i)$ introduced in the proof of Lemma \ref{lem-6.4-R2614}, see  \eqref{eq-3.7-210531}.
By the definition of  $\e_2$,  we have
\begin{align*}
 \|\big(\e_2({\bf u}_1) -\e_2({\bf u}_2) \big) (t) \|_{C^{\ga-2}} 
\lesssim & \|\big(\e_2({\bf u}_1) -\e_2({\bf u}_2) \big) (t) \|_{C^{\b+ \ga-2}}  \\
\lesssim  & \|b( u_1) - b( u_2)\|_{\mathcal{L}_T^\beta}
\| u_1^\sharp(t)\|_{C^{\ga}}  
+\| b(u_2) \|_{\mathcal{L}_T^\b} 
 \|\big(u_1^\sharp -u_2^\sharp \big)(t) \|_{C^{\ga}}. \notag
\end{align*}
Therefore, the estimate \eqref{6.16-R2613} is immediately obtained by 
Lemma \ref{lem-6.2-R2616} above and (3.17) of \cite{FHSX}.
\end{proof}
As the final preparation, we give the estimate on the remainder $\zeta$ in \eqref{eq:2.18-F}.
\begin{lem} \label{lem-6.9-R267}
We have that $\zeta$ is locally Lipschitz continuous in  $(u_0, u', u^\sharp) \in \mathcal{B}_T(\lambda)$. In fact, we have the following estimate
\begin{align*}  
 \sup_{0<t \le T} t^{\frac{\ga -\a}2} 
\|\big(\zeta({\bf u}_1) -\zeta({\bf u}_2) \big)(t)\|_{C^{\a +\b -2}} 
\lesssim & K(\|{\bf u}_1\|_{\a, \b, \ga}, \|{\bf u}_2\|_{\a, \b, \ga})
\tilde{K}_2(X, \xi)\|{\bf u}_1 -{\bf u}_2\|_{\a, \b, \ga}   \notag\\
&  +
 K_0(\|u_0^{1} \|_{C^{\a }}) \|{\bf u}_1\|_{\a,\b,\ga} 
\|\xi\|_{C^{\a-2}} \|u_0^{1}-u_0^{2} \|_{C^{\a}} ,   \notag  
\end{align*}
where $\tilde{K}_2(X, \xi)$ is the constant introduced in Proposition 3.7 of \cite{FHSX}, which depends on $\|\xi\|_{C^{\a -2}}$ and $\|\Pi(\nabla X, \xi)\|_{C^{2\a -3}}$.
\end{lem}

\begin{proof}
This can be shown by similar arguments to Proposition 3.7 of
 \cite{FHSX}. Although many terms should be dealt with, from the
proof of Proposition 3.7 of \cite{FHSX}, we only have to reevaluate 
the terms involving $u_0^T$, that is, $\Pi(a(\nabla u_0^T), \bar\Pi_{u'}\xi)$  and  the term $A_2$ defined by
 (2.27) of \cite{FHSX}. Since we use the norm of 
$\|{\bf u}_1 -{\bf u}_2\|_{\a, \b, \ga}$,  the initial value will only affect our estimate when  the estimates of 
$\|b(u_i)\|_{\mathcal{L}_T^\b}$ and  $\|b(u_1) - b(u_2) \|_{\mathcal{L}_T^\b}$ are used in the proof of Proposition 3.7 of \cite{FHSX}, where $b(u_i)$ is defined by \eqref{eq-3.7-210531}.
 However, the terms $\Pi(a(\nabla u_0^T), \bar\Pi_{u'}\xi)$ and $A_2$ can be handled by the analogous arguments to Lemma \ref{lem-6.6-R2616} above, and Lemma 4.3 of \cite{FHSX}. 
Noting $u' \in \mathcal{L}_T^\b$, a similar argument to \eqref{6.12-R2613}
yields that
\begin{align*}
 & \sup_{0< t \leq T} t^{\frac{\ga-\a}{2}} \Big\| \Big(\Pi(a(\nabla u_0^{1,T} ), \bar\Pi_{u_1'}\xi) -\Pi(a(\nabla u_0^{2, T} ), \bar\Pi_{u_1'}\xi) \Big)(t) \Big\|_{C^{\a+ \b -2}}  \\
\lesssim &(1+\|u_0^{1} \|_{C^{\a }}) \|u_1'\|_{\mathcal{L}_T^\b}\|\xi\|_{C^{\a-2}} \|u_0^{1}-u_0^{2} \|_{C^{\a}} . \notag
\end{align*}
Then, recalling  the estimate deduced in the proof of Lemma 4.1 of \cite{FHSX} and the bilinearity of the resonant term, we easily have
\begin{align*}
 & \sup_{0< t \leq T} t^{\frac{\ga-\a}{2}}  \Big\| \Big(\Pi(a(\nabla u_0^{1,T} ), \bar\Pi_{u_1'}\xi) -\Pi(a(\nabla u_0^{2, T} ), \bar\Pi_{u_2'}\xi) \Big)(t)
\Big\|_{C^{\a+ \b -2}}  \\
\lesssim &  K_0(\|u_0^{1} \|_{C^{\a }}) \|{\bf u}_1\|_{\a,\b,\ga}
\|\xi\|_{C^{\a-2}}    \|u_0^{1}-u_0^{2} \|_{C^{\a}} . \notag
\end{align*}
On the other hand, $A_2$ can be essentially evaluated by the arguments for Lemma 4.3 of \cite{FHSX} thanks to Lemma \ref{lem-6.2-R2616}. Here, we give the estimate on the first term of $A_2$ as an example.  By bilinearity of $R$, we have
\begin{align*}
& \|R( b( u_1), u_1';\xi)
-R(  b ( u_2), u_2';\xi)
\|_{C_TC^{ \a+\b  -2}}
\\
\lesssim & \|R( b( u_1) -b(u_2), u_1';\xi)
\|_{C_TC^{ \a +\b -2 }}
+ \|R( b( u_2), u_1'-u_2';\xi)
\|_{C_TC^{  \a +\b  -2}} \notag\\
\lesssim &  K(\|{\bf u}_1\|_{\a, \b, \ga}, \|{\bf u}_2 \|_{\a, \b, \ga})
(1+\|\xi\|_{C^{\a-2}})^2\|\xi\|_{C^{\a -2}}   \|{\bf u}_1 -  {\bf u}_2 \|_{\a, \b, \ga}  \\
& + K_0(\|u_0^1\|_{C^\a})  \|{\bf u}_1\|_{\a,\b,\ga}
\|\xi\|_{C^{\a-2}}  \| u_0^1 -u_0^2 \|_{C^\a},
\notag 
\end{align*}
where  Lemma \ref{lem-6.2-R2616} and $\a-\b-1>0$ have been used for the last inequality.
Consequently, we complete  the proof.
\end{proof}

In the end, let us give the proof of Theorem \ref{thm:ContInintial} based on the above lemmas.
\begin{proof}[Proof of Theorem \ref{thm:ContInintial}]

According to the proof of Theorem 3.1 of \cite{FHSX}, we also know that the choice of $T$ depends locally uniformly on $\|\tilde{u}_0 -u_0\|_{C^{\a}}$ according to the estimates (3.48) and (3.50) of \cite{FHSX}. In fact, noting that constants $K(\|{\bf u}\|_{\a, \b, \ga})$ in (3.48)  of \cite{FHSX}  and $K(\|{\bf u}_1\|_{\a, \b, \ga}, \|{\bf u}_1\|_{\a, \b, \ga})$ in (3.50) of \cite{FHSX} can be controlled by some polynomial,     we can choose the time $T$ as the function $T=T(r)$  for all $\|\tilde{u}_0 -u_0\|_{C^\a} < r$ with a small enough $r>0$. 
In particular,  we know that $\Phi$ is contractive on $\mathcal{B}_T(\lambda)$ for a large enough $\lambda$ and a small enough $T>0$, that is,  there exists $\kappa \in (0,1)$ such that 
$$\|\Phi({u}_0,u'_1,u_1^\sharp) -\Phi({u}_0, u'_2,u_2^\sharp)\|_{\a, \b,\ga}  
\leq \kappa \|{\bf u}_1 -{\bf u}_2\|_{\a, \b,\ga}, \ (u_0, {\bf u}_i ) \in \mathcal{B}_T(\la).$$
Note that $u_0$ in $\Phi$ is the same in this estimate.
Therefore, thanks to the implicit function theorem, it is enough for us to show the continuity of the map $\Phi$ on  $\mathcal{B}_T(\lambda)$. 
In fact,  we can show that the map $\Phi$ is locally Lipschitz
continuous on $\mathcal{B}_T(\la)$.
Let us take two elements $(u_0^i, {\bf u}_i) \in  \mathcal{B}_T(\la), \ i=1,2$ and denote their images by $(u_0^i, v_i', v_i^\sharp)$ under the map $\Phi$. Then  by \eqref{Phi-6.5}-\eqref{eq:2.18-F},
we have that the different $v_1^\sharp  -v_2^\sharp $ satisfies the equation
\begin{align*}
\mathcal{L}^0_1 v_1^\sharp -\mathcal{L}^0_2 v_2^\sharp= 
& {\Pi}_{a(\nabla u_{0}^{1,T})v_1'} \xi -{\Pi}_{a(\nabla u_{0}^{2,T})v_2'} \xi -
\big(\mathcal{L}_1^0(\bar{\Pi}_{z_1'} X) -\mathcal{L}_2^0(\bar{\Pi}_{z_2'} X) \big) 
\\
&  +
  \e_1({\bf u}_1) -\e_1({\bf u}_2 )
 + \e_2({\bf u}_1)-\e_2({\bf u}_2) +\zeta({\bf u}_1)  - \zeta({\bf u}_2) \notag
\end{align*}
with  the  initial value 
$
(v_1^\sharp- v_2^\sharp)(0)=u_0^1-u_0^2-\Pi_{u_1'(0)-u_2'(0)}X \in {C}^\alpha;
$ 
recall that $v_0^i =u_0^i, v_i'(0)= u_i'(0)=\frac{g(\nabla u_0^i)}{a(\nabla u_0^i)}, \ i=1,2$,
where
$\mathcal{L}_i^0 := \partial_t -a(\nabla u_0^{i, T})\De$.
We easily see that it can be rewritten to the next equation.
\begin{align}  \label{6.7-R2617}
\mathcal{L}^0_2(v_1^\sharp- v_2^\sharp)=& 
{\Pi}_{a(\nabla u_{0}^{2,T})(v_1'-v_2')} \xi  
-   \mathcal{L}_2^0(\bar{\Pi}_{v_1' -v_2'}X) 
+(a(\nabla u_0^{1,T}) -a(\nabla u_0^{2,T}) ) \De v_1^\sharp \\
& 
+\big( a(\nabla u_{0,1}^T) -a(\nabla u_{0,2}^T ) ) \De \bar{\Pi }_{v_1'} X
+{\Pi}_{(a(\nabla u_{0}^{1,T} )- a(\nabla u_{0}^{2,T}) )v_1'} \xi 
\notag \\
&  + \e_1({\bf u}_1) -\e_1({\bf u}_2 ) 
 + \e_2({\bf u}_1)-\e_2({\bf u}_2) +\zeta({\bf u}_1)  - \zeta({\bf u}_2).
\notag
\end{align}

By Lemma 2.10 of \cite{FHSX} and $v_1'-v_2' \in \mathcal{L}_T^\b$, we see the first two terms
of \eqref{6.7-R2617} can  be estimated as follows:
\begin{align} \label{6.8-R2617}
& \sup_{0<t \leq T} t^{\frac{\ga -\a}{2}}
\|\big({\Pi}_{a(\nabla u_{0}^{2,T})(v_1'-v_2')} \xi  
-   \mathcal{L}_2^0(\bar{\Pi}_{v_1' -v_2'}X)  \big)(t)\|_{C^{\a+ \b-2}}
\\
\lesssim
& 
 (1+ \|u_0^2\|_{C^\a} )\|v_1'-v_2' \|_{\mathcal{L}_T^\b}\|X\|_{C^\a}.
 \notag
\end{align}
As we explained in Lemma \ref{lem-6.9-R267}, the estimate  (3.7) in Lemma 3.3 of \cite{FHSX} still holds for $\e_1({\bf u}_1) -\e_1({\bf u}_2)$  in the framework of this theorem. More precisely, we have
\begin{align}\label{6.9-R2617}
& \sup_{0<t\le T} t^{\frac{\ga-\a}2} 
\|\big(\e_1({\bf u}_1) -\e_1({\bf u}_2 )\big)(t) \|_{C^{2\a -3} 
} \\
 \lesssim & K(\|{\bf u}_1 \|_{\a,\b,\ga}, \|{\bf u}_2\|_{\a,\b,\ga})
(1+\|\xi\|_{C^{\a -2}})^2 \|\xi\|_{C^{\a -2}}
\| {\bf u}_1 - {\bf u}_2\|_{\a, \b, \gamma}.  \notag 
\end{align}
Now let us denote by $\phi_1(t)$  the sum of $(a(\nabla u_0^{1,T}) -a(\nabla u_0^{2,T}) ) \De v_1^\sharp(t)$ and $ \big( \e_2({\bf u}_1)-\e_2( {\bf u}_2 ) \big)(t)$,  and  by $\phi_2(t)$ all of the other terms in the right hand side of 
\eqref{6.7-R2617}. Then by \eqref{6.8-R2617}, 
\eqref{6.9-R2617} and 
Lemmas \ref{lem-6.3-R2616}-\ref{lem-6.9-R267}, we easily known that  $\phi_1\in C((0,T], C^{\ga-2})$
and $\phi_2 \in C((0,T], C^{\a+ \b-2})$ satisfy  the assumptions formulated in  Lemma 3.5 of \cite{FHSX}. So we obtain 
\begin{align}\label{eq-3.33-0324}
& \sup_{0<t \le T}  t^{\frac{\ga-\a}2} \|v_1^\sharp(t) 
- v_2^\sharp(t)\|_{C^\ga} 
+\|v_1^\sharp-v_2^\sharp\|_{\mathcal{L}_T^\a}\\
  \lesssim & \|u_0^1-u_0^2\|_{C^\a}
  +\|\Pi_{u_1'(0)-u_2'(0)}X\|_{C^\a} 
  + T^{\frac{\a+\b-\ga}2} (1+\|u_0^2\|_{C^\a}) 
\|v_1'-v_2'\|_{\mathcal{L}_T^\b} \|\xi\|_{C^{\a -2}} \notag \\
&  +K_0(\|u_0^1\|_{C^\a}) \big(\|{\bf u}_1\|_{\a,\b,\ga}+
\|(v_1', v_1^\sharp)\|_{\a,\b,\ga} \big)(1+ \|\xi\|_{C^{\a-2} } )\|u_0^1 -u_0^2\|_{C^\a} \notag \\
& + T^{\frac{\a + \b -\ga}2} K(\|{\bf u}_1\|_{\a, \b, \ga}, \|{\bf u}_2\|_{\a, \b, \ga})
 \tilde{K}_2(X, \xi) \|{\bf u}_1 -{\bf u}_2\|_{\a, \b, \ga}.  \notag
\end{align}
On the other hand, using Lemmas 2.1 and 2.9 of \cite{FHSX}, we have
\begin{align*}
\|\Pi_{u_1'(0)-u_2'(0)}X\|_{C^\a} 
\lesssim & \left\| \frac{1}{a(\nabla u_0^1) a(\nabla u_0^2)}\right\|_{C^\b}
\|g(\nabla u_0^1)a(\nabla u_0^2) -g(\nabla u_0^2)a(\nabla u_0^1) \|_{C^{\b}}
\|X\|_{C^\a}   \\
\lesssim &  K_0(\|u_0^1\|_{C^\a}, \|u_0^2\|_{C^\a})\|u_0^1 -u_0^2\|_{C^\a} 
\|\xi\|_{C^{\a-2} }.
\end{align*}
Therefore,  thanks to  Lemmas \ref{lem-6.3-R2616} and \ref{lem-6.4-R2614}, by \eqref{eq-3.33-0324} together with $0< \a+ \b-\ga< \a-\b-1$, 
we have that 
\begin{align*}
&\|\Phi(u_0^1, u_1', u_1^\sharp ) -\Phi(u_0^2, u_2', u_2^\sharp )\|_{\a, \b, \ga}  \\
  \lesssim &
 K(\|{\bf u}_1\|_{\a, \b, \ga})\big( \tilde{K}_1(X, \xi)  +K_0(\|u_0^1\|_{C^\a}, \|u_0^2\|_{C^\a}) \big)
(1+\|\xi\|_{C^\a}) \|u_0^1 -u_0^2\|_{C^\a} \notag \\
& + T^{\frac{\a + \b -\ga}2} K(\|{\bf u}_1\|_{\a, \b, \ga}, \|{\bf u}_2\|_{\a, \b, \ga})  \tilde{K}_2(X, \xi) \|{\bf u}_1 -{\bf u}_2\|_{\a, \b, \ga}  \notag
\end{align*}
which implies the local Lipschitz continuity of $\Phi$ on $\mathcal{B}_T(r, \la)$. Consequently, the proof is completed.
\end{proof}
\begin{rem}\label{rem-3.3}
In the proof of Proposition \ref{prop:2.6-B}, we used the continuity of the solutions in $ (m, u_0, \hat{\xi})$ of the SPDE
\begin{equation}  \label{eq:1.1-210530}
\partial_t u = a(\nabla u +m)\De u + g(\nabla u +m)\cdot\xi,
\end{equation}
where $m\in \R$. Although we took $m=0$ and fixed it in Theorem \ref{thm:ContInintial},  the continuity of the solutions in $(m, u_0,  \hat{\xi})$ of the SPDE \eqref{eq:1.1-210530} can be shown by similar arguments. Roughly speaking, instead of the map $\Phi$ defined by \eqref{Phi-6.5}, it is natural to study the map $\Phi: (m, u_0, u', u^\sharp) \mapsto (v', v^\sharp)$,
where  $v', v^\sharp$ are determined  by \eqref{eq:2.17-F}, \eqref{eq:2.18-F} by replacing $a(\cdot), g(\cdot)$ by $a(\cdot+m), g(\cdot+m)$ respectively.
Then, all of the estimates in Lemmas \ref{lem-6.2-R2616}-\ref{lem-6.9-R267}
still hold if we replace $\|{\bf u}_i\|_{\a, \b, \ga},  \|{\bf u}_1 -{\bf u}_2\|_{\a, \b, \ga}, \| u_0^i\|_{C^\a},  \| u_0^1 -u_0^2 \|_{C^\a} $ by $\|{\bf u}_i\|_{\a, \b, \ga}+|m^i|, \|{\bf u}_1 -{\bf u}_2\|_{\a, \b, \ga} +|m^1- m^2|, \| u_0^i\|_{C^\a} +|m^i|,   \| u_0^1 -u_0^2 \|_{C^\a} +|m^1- m^2|$ respectively. 
For example, if $(m^i, u_0^i, u_i', u_i^\sharp), i=1,2$ are given, then instead of \eqref{eq-6.11-R2616} in Lemma \ref{lem-6.2-R2616}, we have 
\begin{align*}
& \left\|a(\nabla u_1+ m^1) -a(\nabla u_{0}^{1,T} +m^1)  -\Big(a(\nabla u_2 +m^2) -a(\nabla u_{0}^{2,T} +m^2) \Big)  \right\|_{\mathcal{L}_T^\b} \\
\lesssim &  
T^{\frac{\a-\b-1}{2}} K(\|{\bf u}_1\|_{\a, \b, \ga}, |m_1|) (1+\|\xi \|_{C^{\a-2} })^2
 (\|{\bf u}_1 -{\bf u}_2\|_{\a, \b, \ga} + |m^1 -m^2|) \\
&  +K_0(\|u_0^1\|_{C^\a}, |m_1| )(\| u_0^1 -u_0^2 \|_{C^\a} + |m^1 -m^2|), \notag
\end{align*} 
because we just replaced the functions $a(\nabla u_i), a(\nabla u_{0}^{i,T} ) $  by $a(\nabla u_i + m^i), a(\nabla u_{0}^{i,T} +m^i)$ and $a(\cdot + m)$ satisfies \eqref{eq:acC} for all $m\in \R$. Consequently, we can testify that $\Phi$ satisfies the conditions of  the implicit function theorem by the analogous arguments to the proof of Theorem \ref{thm:ContInintial} and therefore show the desired result.

One can interpret the continuity in $m$ as that in the boundary condition, modified
as in (1.6) of \cite{FHSX}, for the original SPDE \eqref{eq:1.1}.
\end{rem}

\section*{Acknowledgements}
T.\ Funaki was supported in part by JSPS KAKENHI, Grant-in-Aid for Scientific Researches (A) 18H03672 and (S) 16H06338, and 
B.\ Xie was supported in part by JSPS KAKENHI, Grant-in-Aid for Scientific Research (C)  20K03627.

\end{document}